\RequirePackage{fix-cm}

\documentclass[11pt,leqno,a4paper,oneside]{amsart}

\usepackage{amssymb,amsthm,enumerate}
\usepackage{wasysym}
\usepackage{array}
\usepackage{enumitem}
\usepackage[pdftex,bookmarks=true]{hyperref}                        
\usepackage[capitalise,nameinlink]{cleveref}
\usepackage[all]{xy}

\usepackage{tikz}                                
\usetikzlibrary{angles}
\usetikzlibrary{decorations.markings,calc}
\usetikzlibrary{cd}
\usepackage{longtable}
\usepackage{mathrsfs}
\usepackage{hhline}
\usepackage{multicol}
\usepackage{multirow}
\usepackage{blkarray} 
\usepackage{lscape}
\usepackage{rotating}
\usepackage[font=sc]{caption}
\usepackage{subcaption}
\usepackage{adjustbox}
\usepackage{comment}
\usepackage{booktabs}
\usepackage{tabularx}
\newcolumntype{Y}{>{\centering\arraybackslash}X}
\usepackage{nicematrix}
\usepackage{doi} 

\input xy
\xyoption{all}

\entrymodifiers={!!<0pt,0.7ex>+}  

\hypersetup{pdffitwindow=true,
colorlinks=true,
citecolor=black,
filecolor=black,
linkcolor=black,
urlcolor=black,
hypertexnames=true}

\setlist{itemsep=1mm, topsep=2mm, leftmargin=10mm}

\newcommand{\IF}{{\mathbb{F}}}

\newcommand{\IZ}{{\mathbb{Z}}}

\newcommand{\cO}{{\mathcal{O}}}

\newcommand{\Iso}{\textup{MetaFrob}}
\newcommand{\propernormalsubgroup}{%
\mathrel{\ooalign{$\lneq$\cr\raise.22ex\hbox{$\lhd$}\cr}}}

\DeclareMathOperator{\Aut}{Aut}                  
\DeclareMathOperator{\Syl}{Syl}                  
\DeclareMathOperator{\Res}{Res}               
\DeclareMathOperator{\Ind}{Ind}                  
\DeclareMathOperator{\Inf}{Inf}                    
\DeclareMathOperator{\tr}{tr}			
\DeclareMathOperator{\Ker}{ker}
\DeclareMathOperator{\Irr}{Irr}
\DeclareMathOperator{\IBr}{IBr}
\DeclareMathOperator{\Lin}{Lin}

\DeclareMathOperator{\TS}{TS}

\DeclareMathOperator{\Triv}{Triv}

\DeclareMathOperator{\GL}{\operatorname{GL}}
\DeclareMathOperator{\SL}{\operatorname{SL}}
\DeclareMathOperator{\PSL}{\operatorname{PSL}}

\DeclareMathOperator{\AGL}{\operatorname{AGL}}
\newcommand{\TrivialGroup}{\{1\}}
\newcommand{\decmat}{\textup{Dec}_p}

\let\lra=\longrightarrow

\newtheoremstyle{thmnew}{3ex}{3ex}{\itshape}{}{\bf}{.}{.5em}{}
\theoremstyle{thmnew}
\newtheorem{thm}{Theorem}[section]
\newtheorem{lem}[thm]{Lemma}

\newtheorem{prop}[thm]{Proposition}

\newtheoremstyle{defnew}{3ex}{3ex}{}{}{\bf}{.}{.5em}{}
\theoremstyle{defnew}
\newtheorem{exmp}[thm]{Example}

\newtheorem{conv}[thm]{Convention}

\newtheorem{nota}[thm]{Notation}

\theoremstyle{remark}
\newtheorem{rem}[thm]{Remark}

\begin{document}

	
        \title[Trivial source character tables of Frobenius groups of type $(C_p\times C_p)\rtimes H$]{Trivial source character tables of Frobenius groups of type $(C_p\times C_p)\rtimes H$}

        \dedicatory{Dedicated to the memory of Richard Parker}
	
        \author{{Bernhard B\"ohmler and Caroline Lassueur}}
        \address{{\sc Bernhard B\"ohmler},  Leibniz Universit\"at Hannover, Institut f\"ur Algebra, Zahlentheorie und Diskrete Mathematik, Welfengarten 1, 30167 Hannover, Germany}
        \email{boehmler@math.uni-hannover.de}
        \address{{\sc Caroline Lassueur},  Leibniz Universit\"at Hannover, Institut f\"ur Algebra, Zahlentheorie und Diskrete Mathematik, Welfengarten 1, 30167 Hannover, Germany}
        \email{lassueur@math.uni-hannover.de,lassueur@mathematik.uni-kl.de}

        \date{\today}

        \begin{abstract}
        Let $p$ be a prime number. We compute the trivial source character tables 
        of finite Frobenius groups $G$ with an abelian Frobenius complement $H$ and
        an elementary abelian Frobenius kernel of order~$p^2$. More precisely, we deal with infinite families of such groups which occur in the two extremal cases for the fusion of $p$-subgroups: the case in which there exists exactly one $G$-conjugacy class of non-trivial cyclic $p$-subgroups, and the case in which there exist exactly~$p+1$ distinct $G$-conjugacy classes of non-trivial cyclic $p$-subgroups.
	\end{abstract}

	\thanks{
		The authors gratefully acknowledge financial support by the DFG-SFB/TRR195.
	}
	
	\maketitle

	
\pagestyle{myheadings}
\markboth{B. B\"ohmler and C. Lassueur}{Trivial source character tables of Frobenius groups of type $(C_p\times C_p)\rtimes H$}

\section{Introduction}
    Let $G$ be a finite group. Let $p$ be a prime number dividing the order of $G$ and let $k$ be a large enough field of characteristic $p$. Permutation $kG$-modules and their direct summands -- called \emph{$p$-permutation modules} or also \emph{trivial source modules} -- are omnipresent in the modular representation theory of finite groups. They are, for example, elementary building blocks for the construction and for the understanding of different categorical equivalences between block algebras, such as splendid Rickard equivalences, $p$-permutation equivalences, source-algebra equivalences, or Morita equivalences with endo-permutation source. A deep understanding of the structure of these modules is therefore essential.
    \par
    In this manuscript, we go back to ideas of Benson and Parker developed in~\cite{BP}.  
    Any trivial source $kG$-module can be lifted to characteristic zero and affords a well-defined ordinary character, which contains essential information about its structure. 
    The \emph{trivial source character table $\Triv_p(G)$ of $G$ at the prime $p$} collects this information in a table; it is the \emph{species table} or \emph{representation table} of the trivial source ring in the sense of~\cite{BP,BensonBookOld,BensonBookI}.  More precisely, it provides us with information about the character values of all the indecomposable trivial source $kG$-modules and their Brauer quotients at all $p'$-conjugacy classes. See  Subsection~\ref{ssec:DefTSCT}
    for a precise definition.
    \par
    The present article is in fact part of a program aiming at gathering information about trivial source modules of small finite groups and their associated \emph{trivial source character tables} in a database~\cite{DatabaseTSCTs}. Isolated examples --  calculated by Benson, and Lux and Pahlings --  can be found in ~\cite[Appendix]{BensonBookOld} and~\cite[\S4.10]{LuxPah}.  
    More recently, the first author, as part of his doctoral thesis~\cite{BBthesis}, developed  \textsf{GAP4}~\cite{gap} and \textsf{MAGMA}~\cite{magma}  algorithms, which could be used to compute the trivial source character tables of finite groups of order less than 100, as well as the  trivial source character tables of various small (non-abelian) quasi-simple groups. The latter algorithms rely, in particular, on the MeatAxe algorithm, first introduced by R. Parker. 
    Meanwhile, in~\cite{BBCLNF} and~\cite{FarLas23} the authors and Farrell computed \emph{generic} trivial source character tables for the groups $\SL_2(q)$ and $\PSL_2(q)$ in cross characteristic, using the generic character tables of these groups and theoretical methods involving block theory. 
    \par
    Using the data produced in \cite{DatabaseTSCTs} we identified some interesting families of finite groups, of which we can calculate the trivial source character tables from a purely theoretical point of view. In this regard, the main results of this article consist in the calculation of the trivial source character tables at $p$ of the following infinite families of  Frobenius groups with an abelian Frobenius complement $H$ and elementary abelian Frobenius kernel of rank $2$:
        \begin{enumerate}[label=(\Roman*),itemsep=1ex]
             \item the family of all metabelian Frobenius groups of type $(C_p\times C_p)\rtimes H$, in which there is precisely one conjugacy class of subgroups of order $p$\,; 
             \item the family of all metabelian Frobenius groups of type $(C_p\times C_p)\rtimes H$ in which there are precisely $p+1$ conjugacy classes of subgroups of order $p$\,. 
        \end{enumerate}
    We note that the unique group of type (I) with $|H|=p^2-1$ is $\AGL_1(p^2)$\,.
    Furthermore, if~$p=2$, then the only group of type (I) is the alternating group $\mathfrak{A}_4$, while groups of type (II) only occur for odd prime numbers $p$. For this reason, we exclude the prime number $2$ from all our calculations in this manuscript. 
    We also emphasise that contrary to ~\cite{BBCLNF,FarLas23}, it is not possible to use block theoretical arguments in the present cases, because such groups possess only one $p$-block.
    \par
    The paper is structured as follows.  In~\cref{sec:notanddefs} we introduce our notation and conventions. In~\cref{sec:Frobenius_prelim}, first we review some properties of Frobenius groups and their ordinary and Brauer characters. Then,  we characterise metabelian Frobenius groups with elementary abelian Frobenius kernel. The trivial source character tables are calculated in~\cref{sec:MaximalFusion} for groups of type~(I), respectively in~\cref{sec:NO_FUSION} for groups of type~(II).
\par


\vspace{6mm}

\section{Preliminaries}
\label{sec:notanddefs}

\vspace{2mm}
\subsection{General notation}
    Throughout, unless otherwise stated, we adopt the notation and conventions below. We let~$p$ denote an odd prime number and  $G$ a finite group of order divisible by~$p$. 
    We let~$(K,\cO,k)$ be a $p$-modular system, where~$\cO$ denotes a complete discrete valuation ring of characteristic zero with field of fractions~$K=\text{Frac}(\cO)$ and  residue field $k=\cO/J(\cO)$ of characteristic~$p$. Following \cite[\S17A]{CurtisReinerMethods1}, we assume that $K$ is \emph{sufficiently large (relative to~$G$)}, i.e. $K$ contains all $\exp(G)$-th roots of unity. Then $k$ is also sufficiently large (relative to $G$) and $K$ and $k$ are splitting fields for $G$ and all of its subgroups. 
    For~$R\in\{\cO,k\}$, $RG$-modules are assumed to be finitely generated left $RG$-lattices, that is, free as~$R$-modules, and  we let  $R$  denote the trivial $RG$-lattice. 
    \par
    Given a positive integer~$n$, we let $C_{n}$ denote the cyclic group of order~$n$. We let $\mathbf{O}_p(G)$ denote the largest normal~$p$-subgroup of $G$, $\textup{Syl}_p(G)$ denote the set of all Sylow $p$-subgroups of $G$,  $ccls(G)$ denote a set of representatives for the conjugacy classes of $G$,  $[G]_{p^\prime}$ denote a set of representatives for the $p$-regular conjugacy classes of $G$, and we let $G_{p'}:=\{g\in G\mid p\nmid o(g)\}$. 
    We recall that a group $G$ with a normal subgroup $N$ and a subgroup $H$ is said to be the \emph{internal semi-direct product of $N$ by $H$}, written~$G=N\rtimes H$, provided $G=NH$ and $N\cap H=\{1\}$.
    \par
    Given $H\leq G$, an ordinary character $\psi$ of $H$ and $\chi$ an ordinary character of $G$, we write 
    $\Ind_H^G(\psi)$ for the induction of $\psi$ from $H$ to $G$, 
    $\Res^G_H(\chi)$ for the restriction of $\chi$ from $G$ to $H$, $\chi^\circ:=\chi|_{G_{p'}}$ for the reduction modulo $p$ of $\chi$, and $1_H$ for the trivial character of $H$. Given $N\unlhd G$ and an ordinary character~$\nu$ of $G/N$, we write $\Inf_{G/N}^G(\nu)$ for the inflation of $\nu$ from $G/N$ to $G$. Similarly, we write~$\Ind_H^G(L)$ for the induction of the~$kH$-module $L$ from $H$ to $G$, $\Res^G_H(M)$ for the restriction of the $kG$-module $M$ from~$G$ to $H$, and~$\Inf_{G/N}^G(U)$ for the inflation of the $k[G/N]$-module $U$ from $G/N$ to $G$. Moreover, if~$M$ is a~$kG$-module, then we denote by $\varphi_{M}$ the Brauer character afforded by $M$,  and if $Q\leq G$ then the Brauer quotient~(or Brauer construction) of $M$ at $Q$ is the $k$-vector space $M[Q]:=M^{Q}\big/ \sum_{R<Q}\tr_{R}^{Q}(M^{R})$, where~$M^{Q}$ denotes the fixed points of $M$ under $Q$ and $\tr_{R}^{Q}$ denotes  the relative trace map. This vector space has a natural structure of a $kN_{G}(Q)$-module, but also of a $kN_{G}(Q)/Q$-module, and is equal to zero if $Q$ is not a $p$-subgroup. Moreover, we use the abbreviation PIM to mean a \emph{projective indecomposable module} and we  denote by $\Irr(kG)$ the set of all simple $kG$-modules, considered up to isomorphism. 
    We assume that the reader is familiar with elementary notions of ordinary and modular representation theory of finite groups. We refer to \cite{LinckBook, Webb, NagaoTsushima, HuppertCharacterTheory, CurtisReinerMethods1} for further standard notation and background results.

\vspace{2mm}
\subsection{Character tables and decomposition matrices}\label{ssec:NotaCharTables}
 We let $\Irr(G)$, $\Lin(G)$, and $\IBr_p(G)$ denote the set of all irreducible $K$-characters of~$G$, the set of all linear characters of $G$, and the set of all irreducible~$p$-Brauer characters of $G$, respectively.
 We let
    \[    X(G):=\Big(\chi(g)\Big)_{\substack{\chi\in\Irr(G)\\g\in ccls(G)}}\in K^{|\Irr(G)|\times|ccls(G)|}
    \]
    denote the ordinary character table of $G$ and we let
    \[
    X(G,p'):=\Big(\chi(g)\Big)_{\substack{\chi\in\Irr(G)\\g\in [G]_{p'}}}\in K^{|\Irr(G)|\times|[G]_{{p'}}|}
    \]
    denote the matrix  obtained from $X(G)$ by removing the columns labelled by $p$-singular conjugacy classes. Moreover, we always assume that the first column of these matrices is labelled by the class of $1$.
    We recall that for any $\chi\in\Irr(G)$ there exist uniquely determined non-negative integers~$d_{\chi\varphi}$ such that $\chi^\circ=\sum_{\varphi\in\IBr_p(G)}d_{\chi\varphi}\varphi$\,.
    Then, for any ${\varphi\in\IBr_p(G)}$,  the \emph{projective indecomposable character} associated to $\varphi$ is 
     \begin{equation}
     \Phi_\varphi:=\sum_{\chi\in\Irr(G)}d_{\chi\varphi}\chi\,.
     \end{equation}
    The \emph{$p$-decomposition matrix} of $G$ is then
     \[
     \decmat(G)  :=\Big(d_{\chi\varphi}\Big)_{\substack{\chi\in\Irr(G)\\\varphi\in\IBr_p(G)}} \in K^{|\Irr(G)|\times|\IBr_p(G)|}
     \]
    and the \emph{$p$-projective table} of $G$ is 
    \[
    \Phi_p(G):=\Big(  \Phi_\varphi(x)  \Big)_{\substack{\varphi\in\IBr_p(G)\\x\in [G]_{p'}}} \in K^{|\IBr_p(G)|\times [G]_{p'}}\,,
    \]
 which is the table of Brauer character values of the projective indecomposable $kG$-modules. 
   It follows from the definitions that
   \begin{equation}\label{eq:ProjTable}
   \Phi_p(G)=\decmat(G)^{t}\cdot X(G,p') \,.
   \end{equation} 
  \bigskip

   \noindent Finally, the character tables of finite cyclic groups will play an essential role in our calculations, hence in this case we fix the following labelling of the irreducible characters and conjugacy classes.
   
 \enlargethispage{1cm}  
 
\begin{nota}\label{nota:ordinary_ct_of_cyclic_group}
If $G:= \langle x\ |\ x^m=1\rangle \cong C_m$ is a cyclic group of order $m\geq 1$, then we let $\zeta \in K$ denote a primitive $m$-th root of unity and we write the set of ordinary irreducible characters of~$G$ as~$\Irr(G)=\{\xi_1, \ldots, \xi_m\}$, where
$$\xi_{a}(x^j):= \zeta^{(a-1)j}$$
for each $1\leq a \leq m$ and each $0\leq j\leq m-1$. This yields
$$X(C_{m}) := \Big( \xi_a(x^{j-1})  \Big)_{\substack{1\leq a\leq m\\ 1\leq j \leq m}} = \Big( \zeta^{(a-1)(j-1)}  \Big)_{\substack{1\leq a\leq m\\ 1\leq j \leq m}}\,.$$
\end{nota}

\vspace{2mm}
\subsection{Trivial source character tables}\label{ssec:DefTSCT}
Given $R\in\{\cO,k\}$, an $RG$-lattice $M$ is called a \emph{trivial source} $RG$-lattice if it is isomorphic to an indecomposable\footnote{We emphasise here that some authors use the terminology \emph{trivial source module} to mean a finite direct sum of indecomposable $kG$-modules with a trivial source. We always assume such modules to be indecomposable.}  direct summand of an induced lattice $\Ind_{Q}^{G}(R)$, where $Q\leq G$ is a $p$-subgroup. In addition, if $Q$ is of minimal order subject to this property, then $Q$ is a vertex of~$M$.  
It is clear that, up to isomorphism, there are only finitely many trivial source $RG$-lattices.
\par
It is well-known that any trivial source $kG$-module $M$ lifts in a unique way to a trivial source $\cO G$-lattice $\widehat{M}$ (see e.g.~\cite[Corollary 3.11.4]{BensonBookI}) and we denote by $\chi^{}_{\widehat{M}}$ the $K$-character afforded by $\widehat{M}$. 
Moreover, if $\varphi^{}_{M}$ denotes the Brauer character afforded by $M$, then $(\chi_{\widehat{M}})^\circ=\varphi^{}_M$ (see e.g. \cite[Proposition 5.13.6]{Linckelmann1}).
If~$M$ is a PIM, then $\chi^{}_{\widehat{M}}=\Phi_\varphi$, where $\varphi$ is the Brauer character afforded by the unique simple $kG$-module in the socle of $M$.
\par
We will study trivial source modules vertex by vertex.  Hence, we denote by $\TS(G;Q)$ the set of isomorphism classes of trivial source $kG$-modules with vertex $Q$. We notice that  $\TS(G;\{1\})$ is precisely  the set of isomorphism classes of PIMs of $kG$. 
\par
A $p$-subgroup $Q\leq G$ is a vertex of a trivial source $kG$-module $M$  if and only if $M[Q]$ is a non-zero projective $k\overline{N}_{G}(Q)$-module.
Moreover, if this is the case, then the $kN_{G}(Q)$-Green correspondent $f(M)$ of $M$ is $M[Q]$ (viewed as a $kN_{G}(Q)$-module). Thus, there are \smallskip bijections
\begin{center}
	\begin{tabular}{ccccc}
		$\TS(G;Q)$      & $\overset{\sim}{\longrightarrow}$ &    $\TS(N_{G}(Q);Q)$   &  $\overset{\sim}{\longrightarrow}$ & $\TS(\overline{N}_G(Q);\{1\})$  \\
		$M$      & $\mapsto$     &   $f(M)$      &  $\mapsto$     & $M[Q]$
	\end{tabular}
\end{center}
where the inverse of the second map is given by the inflation from $\overline{N}_G(Q):=N_G(Q)/Q$ to $N_{G}(Q)$. 
\noindent These sets are also in bijection with the set of $p^\prime$-conjugacy classes of $\overline{N}_G(Q)$. 
\par

Next, we let $a(kG,\Triv)$ denote the \emph{trivial source ring} of $kG$, which is defined to be the subring of the Green ring of~$kG$ generated by the set of all isomorphism classes of trivial source $kG$-modules. Notice that this ring is finitely generated. By definition, the \emph{trivial source character table of the group $G$ at  the prime $p$}, denoted $\Triv_{p}(G)$,  is the species table of the trivial source ring of $kG$. See e.g.~\cite{BP}. 
However, 
we follow~\cite[Section 4.10]{LuxPah} and consider~$\Triv_{p}(G)$ as the block square matrix defined according to the following convention.

\enlargethispage{1cm}

\begin{conv}\label{conv:tsctbl}%
First, fix  a set of representatives $Q_1,\ldots, Q_r$ ($r\in\IZ_{\geq 1}$) for the conjugacy classes of~$p$-subgroups of $G$ where $Q_{1}:=\{1\}$, $Q_{r}\in\Syl_{p}(G)$ and $|Q_1|\leq \ldots \leq |Q_r|$. For each $1\leq v\leq r$ set~$\overline{N}_{G}(Q_v):=N_{G}(Q_{v})/Q_{v}$.  
	For each pair $(Q_{v},s)$ with $1\leq v\leq r$ and $s\in [\overline{N}_{G}(Q_v)]_{p^\prime}$ there is a ring homomorphism  
	\begin{center}
		\begin{tabular}{cccl}
			$\tau_{Q_{v},s}^{G}$\,:            &   $a(kG,\mbox{Triv})$      & $\lra$ &    $K$     \\
			&   $[M]$      & $\mapsto$     &   $\varphi^{}_{M[Q_{v}]}(s)$   
		\end{tabular}
	\end{center}
	mapping the class of a trivial source $kG$-module $M$ to the value at~$s$ of the Brauer character~$\varphi^{}_{M[Q_{v}]}$ of the Brauer quotient $M[Q_{v}]$. 
	(Note  that the group $G$ acts by conjugation on the pairs $(Q_{v},s)$ and the values of $\tau_{Q_{v},s}^{G}$ do not depend on the choice of  $(Q_{v},s)$ in its $G$-orbit.)
	Then, for each $1\leq i,v\leq r$ define a matrix
	$$T_{i,v}:=\Big( \tau_{Q_{v},s}^{G}([M])\Big)_{\substack{M\in \TS(G;Q_{i})\\s\in [\overline{N}_{G}(Q_v)]_{p^\prime}}}\,.$$
	The \emph{trivial source character table of  $G$ at the prime $p$} is then the block matrix 
	$$\Triv_{p}(G):=\Big[T_{i,v}\Big]_{\substack{1\leq i\leq r\\1\leq v\leq r}}\,.$$
    For convenience, we will label the columns of $\Triv_p(G)$ by representatives of the $p^\prime$-elements of $\overline{N}_G(Q_v)$ in $N_G(Q_v)$ ($1\leq v\leq r$). This is possible e.g. by \cite[Lemma 3.1.12]{BBthesis}. Moreover, we label the rows of $\Triv_{p}(G)$  with the ordinary characters $\chi^{}_{\widehat{M}}$ instead of the isomorphism classes of  trivial source $kG$-modules $M$ themselves.
\end{conv}

\noindent In order to calculate the entries of $\Triv_{p}(G)$, we use the following two well-known lemmata.
\noindent The first one lets us describe certain blocks of the trivial source character table using ordinary and Brauer characters. The second lemma characterises trivial source modules with maximal vertices when a Sylow $p$-subgroup is normal.

\begin{lem}\label{lem:PropTrivpG}
  Let $\Triv_{p}(G)=[T_{i,v}]_{1\leq i,v\leq r}$ be the trivial source character table of the finite group~$G$ at $p$. Then, the following assertions hold:
    \begin{enumerate}[label=\rm(\alph{enumi}),itemsep=1ex]
        \item $T_{i,v}=\mathbf{0}$ if $Q_v\not\leq_G Q_i$, so in particular $T_{i,v}=\mathbf{0}$ for every $1\leq i<v\leq r$\,;
        \item  $T_{i,i}= \Phi_p(\overline{N}_G(Q_i))= \decmat(\overline{N}_G(Q_i))^{t}\cdot X(\overline{N}_G(Q_i),p')$ for every $1\leq i\leq r$\,;
        \item  $T_{i,1}=\big( \chi^{}_{\widehat{M}}(s)\big)_{M\in \TS(G;Q_{i}), s\in [G]_{p^\prime} }$ for every $1\leq i\leq r$\,.
     \end{enumerate}
\end{lem}

\begin{proof}
Assertion (a) is given by \cite[Lemma 4.10.11(b)]{LuxPah}. The first equality in assertion (b) is given by \cite[Lemma 4.10.11(c)]{LuxPah} and the second equality follows from Equation~(\ref{eq:ProjTable}) above. 
Now, if  $v=1$ and $1\leq i\leq r$, then $M[Q_{v}]=M[\{1\}]=M$, so $\tau_{\{1\},s}^{G}([M])=\varphi^{}_{M}(s)=\chi^{}_{\widehat{M}}(s)$  for every $M\in \TS(G;Q_{i})$ and every $s\in [G]_{p'}$, proving assertion~(c). 
\end{proof}

\begin{lem}\label{lem:Simples_G_mod_Op(G)}%
Assume $G$ is a finite group with a normal Sylow $p$-subgroup $P\unlhd G$ such that $G/P$ is an abelian $p'$-group. Then, the following assertions hold:
    \begin{enumerate}[label={\rm(\alph*)},itemsep=1ex]
        \item $\TS(G;P)=\Irr(kG)=\{\Inf_{G/P}^{G}(S)\mid S\in \Irr(k[G/P])\}$, which is the set of all {$1$-dimensional} $kG$-modules (considered up to isomorphism); 
         \item  $\{\chi^{}_{\widehat{M}}\mid M\in\TS(G;P)\}=\Inf_{G/P}^{G}(\Irr(G/P))\subseteq \Lin(G)$\,.
    \end{enumerate}
\end{lem}

\begin{proof}
\begin{enumerate}[label=(\alph*)]
    \item  Clearly  $P=\mathbf{O}_p(G)$.
    Thus, it follows from Clifford's theorem that $\Irr(kG)=\{\Inf_{G/P}^{G}(S)\mid S\in \Irr(k[G/P])\}$, which is precisely the set of all $1$-dimensional $kG$-modules as $G/P$ is an abelian $p^\prime$-group  (see e.g. \cite[Corollary 6.2.2]{Webb}). Now, $1$-dimensional $kG$-modules are trivial source modules with vertex $P$ since their restriction to $P$ must be trivial, thus, as $\overline{N}_G(P)=G/P$ is an abelian  $p^\prime$-group,  these already account for the whole of $\TS(G;P)$\,. (See Subsection~\ref{ssec:DefTSCT}.)
    \item The claim is clear from (a) and the fact that $(\chi^{}_{\widehat{M}})^\circ$ coincides with the Brauer character of $M$ for any $M\in\TS(G;P)$\,.
\end{enumerate}

\end{proof}

\enlargethispage{1cm}

\noindent We refer the reader to the survey \cite{LASpPermSuryey} and to our previous paper~\cite[\S 2]{BBCLNF} for further details and  further properties of trivial source modules and trivial source character tables.
However, we mention the following result from the thesis of the first author, which will be crucial. Inbetween, this result has also appeared in \cite[4.1 Theorem]{BoltjeMonteiro}.

\begin{prop}[{}{\cite[Proposition 3.1.15]{BBthesis}}]\label{Proposition1AbelianGroupsTSModules}
Assume $G$ is a finite group with a normal Sylow $p$-subgroup $P\unlhd G$ such that $G/P$ is abelian. Let $Q$ be a $p$-subgroup of $G$. Then, we have  $P\cap N_G(Q) = \mathbf{O}_p(N_G(Q))\in\textup{Syl}_p(N_G(Q))$ and by the Schur--Zassenhaus Theorem, we may choose a complement $C$ of $P\cap N_G(Q)$ in $N_G(Q)$. 
Let $S$ be a simple $kC$-module, viewed as a simple $k[QC/Q]$-module via the canonical isomorphism $QC/Q\cong C$. Set
        $$
        L:=\Ind_{QC/Q}^{\overline{N}_G(Q)}(S)
        \qquad\text{ and }\qquad 
        U:=\Inf_{\overline{N}_G(Q)}^{N_G(Q)}(L)\,.
        $$
Then, the following assertions hold:
	\begin{enumerate}[label=\rm(\alph{enumi}),itemsep=1ex]
            \item $L$ is a projective indecomposable $k\overline{N}_G(Q)$-module; and  
            \item $M:= \Ind_{N_G(Q)}^{G}(U)$ is indecomposable, hence a trivial source $kG$-module with vertex~$Q$.
	\end{enumerate}
In particular, any element of $\TS(G;Q)$ can be obtained in this way.
\end{prop}


\vspace{6mm}

\section{Background material on Frobenius groups}
\label{sec:Frobenius_prelim}

We start by reviewing basic definitions and results about the character theory of Frobenius groups.

\subsection{Frobenius groups} 

Recall that a finite group  $G$ admitting a non-trivial proper subgroup~$H$ such that
$$H\cap gHg^{-1} = \{1\}$$
for each $g\in G\setminus H$ is called a \emph{Frobenius group} with \emph{Frobenius complement} $H$ (or a \emph{Frobenius group with respect to} $H$). Frobenius proved that in such a group there exists a uniquely determined normal subgroup $F$ such that $G$ is the internal semi-direct product of $F$ by $H$ (i.e. $G=FH$ and~$F\cap H= \{1\}$); concretely, 
$$F=\{1\}\cup \bigg( G\setminus \bigcup_{g\in G} {gHg^{-1}}\bigg).$$
The normal subgroup $F$ is called the \emph{Frobenius kernel} of $G$. See e.g.~{\cite[\S14A]{CurtisReinerMethods1}}. In the sequel, we write Frobenius groups with respect to $H$ as $F\rtimes H$.
We will use the following well-known properties.
 
\begin{lem}\label{lem:Frobenius_General_Properties}
Let $G$ be a Frobenius group with Frobenius complement $H$ and Frobenius kernel~$F$. Then the following assertions hold.
    \begin{enumerate}[label=\rm(\alph{enumi}),itemsep=1ex]
        \item If $H$ is abelian, then $H$ is cyclic.
        \item The integer $|H|$ divides $|F| - 1$. In particular $|G: F|$ and $|F|$ are coprime integers, hence~$F$ is characteristic in $G$.\label{lem:F_characteristic_in_G}
        \item For each $f\in F\setminus\{1\}$ we have $C_G(f)\leq F$.\label{lem:CentraliserFrobenius}
        \item 
        The commutator subgroup of $G$  is $[G,G]=F\rtimes [H,H]$.
    \end{enumerate}
\end{lem}

\begin{proof}
Assertion (a) is due to Burnside and follows directly from \cite[16.7 Theorem b)]{HuppertCharacterTheory}. Assertion (b) is given by {\cite[16.6 Lemma a)]{HuppertCharacterTheory}}. Assertion (c) is given by {\cite[(14.4) Proposition~(i)]{CurtisReinerMethods1}}.
For Assertion (d), first it is clear that $[G,G]\leq F\rtimes [H,H]$ as $G/(F[H,H])\cong H/[H,H]$ is abelian. To prove the reverse inclusion, it suffices to prove that $F=[F,H]$, since then $F[H,H]=[F,H][H,H]\leq [G,G]$. Now, as $G$ is a Frobenius group,  $H\neq 1$ and there exists $h\in H\setminus\{1\}$. Then, the $|F|$ elements of the set $\{fhf^{-1}h^{-1}\in G\mid f\in F\}$ are pairwise distinct. Else $f_1hf_1^{-1}h^{-1}=f_2hf_2^{-1}h^{-1}$ with $f_1\neq f_2\in F$ implies that $h\in C_G(f_2^{-1}f_1)$, contradicting~(c). Finally, as $F\unlhd G$, any element in this set is in $F$, proving that $F\leq [F,H]\leq~F$.
\end{proof}

\subsection{Characters of Frobenius groups}

\noindent The ordinary characters of Frobenius groups are well-known and given by the following theorem.

\begin{thm}[{\cite[(14.4) Proposition]{CurtisReinerMethods1}}]\label{prop:Ordinary_Irreducible_Character_Of_Frobenius_Groups}
	Let $G$ be a Frobenius group with Frobenius complement $H$ and Frobenius kernel~$F$. Then
 $$\Irr(G)=\{\Inf_{G/F}^G(\psi)\ |\ \psi\in \Irr(G/F)\} \sqcup \{\Ind_{F}^{G}(\nu)\ |\ \nu \in T \}\,,$$
 where $T$ is a set of representatives for the orbits of the action of $G$ by conjugation on $\Irr(F)\setminus\{1_F\}$. 
\end{thm}

\noindent Notice that the first set is precisely the set of all irreducible characters of $G$ which contain $F$ in their kernels.

\begin{prop}\label{prop:decompositionMatrix_Frobenius_groups_all_cases_relevant_for_us}
    Let $G$ be a Frobenius group with cyclic Frobenius complement $H\cong C_m$ for some integer~$m\geq 2$ and abelian Frobenius kernel $F$ of order $p^r$ for some positive integer~$r\geq 1$. Then the following assertions hold:
 \begin{enumerate}[label=\rm(\alph{enumi}),itemsep=1ex]
     \item 
     $\Irr(G)=\{\chi_1, \ldots , \chi_{m+\frac{p^r-1}{m}}\}$ where  for each $1\leq a\leq m$ we set 
     $$ \chi_a := \Inf_{G/F}^{G}(\xi_a)$$
     with  $\xi_a\in\Irr(C_m)$  as defined in \cref{nota:ordinary_ct_of_cyclic_group}, and  
     $$\{\chi_{m+1},\ldots , \chi_{m+\frac{p^r-1}{m}}\} = \{\Ind_{F}^{G}(\nu)\ |\ \nu \in T \}$$
     where $T$ is a set of representatives for conjugation action of $G$ on  $\Irr(F)\setminus \{1_F\}$\,;
     \item 
     $\Lin(G)=\Inf_{G/F}^{G}(\Irr(G/F))=\{\chi_1,\ldots,\chi_m\}$ and $\IBr_p(G) = \{\varphi_1, \ldots , \varphi_m\}$
     where  $\varphi_a := \chi_a^{\circ}$ for each $1\leq a\leq m$\,;
     \item $H$ is a set of representatives of the $p$-regular conjugacy classes of $G$;
    \item 
    $\chi_{a}^\circ=\sum_{j=1}^{m} \varphi_j$ for each $1\leq a\leq \frac{p^r-1}{m}$\,.
 \end{enumerate}
\end{prop}

\begin{proof}
Recall from \cref{lem:Frobenius_General_Properties} that $\textup{gcd}(m,p)=1$. 
\begin{enumerate}[label=\rm(\alph{enumi}),itemsep=1ex]
    \item First, it is clear from \cref{prop:Ordinary_Irreducible_Character_Of_Frobenius_Groups} 
    that $G$ has $m$ pairwise distinct ordinary irreducible characters which are inflated from $G/F\cong H\cong C_m$ to $G$. Moreover, 
    \begin{equation*}
    \begin{split}
    |\Irr(G)| &= |\Irr(G/F)| + |\{\Ind_{F}^{G}(\nu)\mid \nu \in T \}|= m + \frac{|F|-1}{|H|}= m +\frac{p^r-1}{m}
    \end{split}
    \end{equation*}
    where the last-but-one equality holds by {\cite[18.7 Theorem b)]{HuppertCharacterTheory}}.
    \item The first claim follows from (a) and \cref{lem:Frobenius_General_Properties}(d). Then, by \cref{lem:Simples_G_mod_Op(G)}, we have $|\IBr_p(G)|=|\IBr_p(G/F)|=|H|=m$.  Since by construction 
    $\chi_1^\circ, \ldots, \chi_m^\circ$
    are pairwise distinct linear Brauer characters, they already account for all the irreducible Brauer characters of~$G$. The claim follows.
    \item Assume that $h_1\neq h_2\in H\setminus \{1\}$ are conjugate in $G$, that is, $gh_1g^{-1}=h_2$ for some $g=fh\in FH=G$ with $f\in F$ and $h\in H$. 
    It follows 
    that
    $ fhh_1h^{-1}f^{-1} =  h_2$, which implies that ${h_1}^{-1}\cdot fh_1f^{-1} = {h_1}^{-1}\cdot h_2$  as $H$ is cyclic. Since $H$ acts fixed-point-freely on $F\setminus \{1\}$, we see that ${h_1}^{-1} fh_1\in F\setminus \{f\}$.
    This is a contradiction, as $F\cap H = \{1\}$. The claim follows, as by (b), a set of representatives for the $p$-regular classes of~$G$ has size $|H|$.
    \item It is immediate from part (b) that the first $m$ rows of $\decmat(G)$ are given by the identity matrix of size $m\times m$. Let now $\chi_a\in\Irr(G)$ with $m+1\leq a\leq m+\frac{p^r-1}{m}$. By (a), there exists a character $\nu\in \Irr(F)\setminus \{1_F\}$ such that $\chi_a=\Ind_F^G(\nu)$. By (c), we only need to prove that ${\chi_a^\circ}\big|_{H} = \big( \sum\limits_{i=j}^{m} {\varphi_j}\big) \big|_{H}$\,. Let $\rho_H$ be the regular character of $H$. Then we have
$$\chi_a^\circ\big|_H = (\Ind_F^G(\nu))^\circ\big|_H = \Res_H^G(\Ind_F^G(\nu))=\nu(1)\cdot \rho_H = \rho_H =\sum_{\xi\in\Irr(H)}\xi = \sum_{j=1}^{m} \varphi_j,$$
where the third equality follows from \cite[18.7 Theorem (b)]{HuppertCharacterTheory} and the last equality follows from (b) and the fact that $H$ is an abelian $p^\prime$-group.
\end{enumerate}
\end{proof}

\subsection{Frobenius groups of type \texorpdfstring{$(C_p\times C_p)\rtimes H$}{[(CpxCp)xH]}}
Recall that $p$ denotes an odd prime number. In this article, the aim is to focus on Frobenius groups~$G$ with cyclic Frobenius complement of order~$m$ and elementary abelian Frobenius kernel of order~$p^2$. In particular, we will compute the trivial source character tables $\Triv_p(G)$ in the following two extremal cases: first the case, in which there is precisely one $G$-conjugacy class of cyclic subgroups of order $p$ (we will call this the \emph{maximal fusion case}); second, the case in which there are precisely $p+1$\ $G$-conjugacy classes of cyclic subgroups of order $p$ (we call this the \emph{minimal fusion case}). In this subsection, we characterise such groups.\par\ \par
\noindent Given integers $m,n > 1$, we denote by $\Iso(m)$ the set of isomorphism classes of metabelian Frobenius groups with Frobenius complement of order $m$ and we set
$$\Iso(m, n):= \{G\ \in \Iso(m)\ \mid \ |G|=mn\}.$$ 
Note that Frobenius groups~$G$ with cyclic Frobenius complement of order~$m$ and elementary abelian Frobenius kernel of order~$p^2$ comprise all elements of $\Iso(m,p^2)$ whose Frobenius kernels are not cyclic.

\begin{lem}\label{LemmaIsoMetabelianFrobeniusSpecialCases}
Let $p$ be an odd prime number and let $m>1$ be an integer such that $m\nmid (p-1)$. Then the following assertions hold:
\begin{enumerate}[label=\rm(\alph{enumi}),itemsep=1ex]
    \item $|\Iso(m,p^2)|=1$;
    \item if $m=p^2-1$ then the unique element of $\Iso(p^2-1,p^2)$ is the affine linear group $\AGL_1(p^2)$ and can be identified with the subgroup
		$$\mathcal{G}:=\left\{\begin{pmatrix}
		a & 0 \\ b & 1\end{pmatrix}\in\GL_2(\IF_{p^2})\ |\ a\in \IF_{p^2}^\times, b\in \IF_{p^2} \right\}$$
		\noindent of $\GL_2(\IF_{p^2})$, which is a Frobenius group with Frobenius kernel and Frobenius complement
		$$\mathcal{F}:=\left\{\begin{pmatrix}
		1 & 0 \\ b & 1\end{pmatrix}\in\GL_2(\IF_{p^2})\ |\ b\in \IF_{p^2} \right\}\ \textup{ and }\ \mathcal{H}:=\left\{\begin{pmatrix}
		a & 0 \\ 0 & 1\end{pmatrix}\in\GL_2(\IF_{p^2})\ |\ a\in \IF_{p^2}^\times\right\}$$
respectively;
    \item if $p< m < p^2-1$ then the unique element of $\Iso(m,p^2)$ can be identified with the subgroup $\widetilde{\mathcal{G}}= \mathcal{F} \widetilde{\mathcal{H}} = \mathcal{F}\rtimes \widetilde{\mathcal{H}}$ of $\mathcal{G}$, where $|\widetilde{\mathcal{H}}|=m$.
\end{enumerate}
\end{lem}

\begin{proof}
\begin{enumerate}[label=\rm(\alph{enumi}),itemsep=1ex]
    \item This follows from the formula given in \cite[Theorem 11.7]{BH98} together with \cite[Remark 11.13.(C)]{BH98}.

    \item
    Assertion (b) is well-known and follows for example from \cite[Exercise 5.15.7]{JacobsonBasicAlgebraII}. 
    \item The group $\widetilde{\mathcal{G}}= \mathcal{F}\rtimes \widetilde{\mathcal{H}}$ is obviously a subgroup of $\mathcal{G}$. Moreover, as $\widetilde{\mathcal{H}} < \mathcal{H}$ and $$\qquad\qquad\widetilde{\mathcal{G}}\setminus \widetilde{\mathcal{H}} = \{f\cdot \widetilde{h}\in\widetilde{\mathcal{G}}\ |\ f\in \mathcal{F}\setminus \{1\}, \widetilde{h}\in \widetilde{\mathcal{H}}\} \quad \subset \quad\mathcal{G}\setminus \mathcal{H} = \{f\cdot h \in \mathcal{G}\ |\ f\in \mathcal{F}\setminus \{1\}, h\in \mathcal{H}\},$$
    it follows from the definition that $\mathcal{F}\rtimes \widetilde{\mathcal{H}}$ is again a Frobenius group.
\end{enumerate}
\end{proof}

\begin{prop}\label{prop:FrobeniusGroups_max_fusion_and_no_fusion}
	Let $G$ be a Frobenius group with Frobenius complement $H\cong C_m$ and Frobenius kernel $F\cong C_p\times C_p$ where $p$ is an odd prime number.
	\begin{enumerate}[label=\rm(\alph*),itemsep=1ex]
		\item 
        The number of $G$-conjugacy classes of subgroups of $G$ of order $p$ equals $1$ if and only if
        $(p + 1)(p-1)_2\mid m$ if and only if~$m=(p + 1)\cdot \mathrm{gcd}(p-1,m)$.
        In this case, we have $|N_H(C)|=\mathrm{gcd}(p-1,m)$ for any subgroup $C\leq G$ of order $p$.
\item 
The number of $G$-conjugacy classes of subgroups of $G$ of order $p$ equals $p+1$ if and only if for any $h\in H\setminus\{1\}$, there exists an integer $1< a(h)\leq p-1$ such that $hfh^{-1}=f^{a(h)}$ for any $f\in F$.
In this case, $H$ is cyclic of order dividing $p-1$.
\end{enumerate}	
\end{prop}

\begin{proof}
Write $F=\langle x \rangle \times \langle y \rangle$ and $H=\langle h \rangle$. As $F$ is elementary abelian of order $p^2$, there are precisely~$p+1$ subgroups of $F$ of order $p$, namely
  $ R_i := \langle x\cdot y^{i}\rangle \quad (0\leq i\leq p-1)\ \text{ and }\ R_{p}:=\langle y \rangle,$
and we let $X := \{R_0,\ldots, R_p\}$.

	\begin{enumerate}[label=\rm(\alph{enumi}),itemsep=1ex]
		\item The equivalent conditions and the fact that $|N_H(C)|=\mathrm{gcd}(p-1,m)$ for any $C\in X$ are proved in  \cite[Example 3.2]{BreuerHethelyiHorvathKuelshammer}\,. 
        It follows from the action of $H$ on $X$ and the induced equality $(p+1)\cdot |N_H(C)|= |H|=m$.

\item 
If the number of $G$-conjugacy classes of subgroups of $G$ of order $p$ equals $p+1$, the action of~$H$ by conjugation must map the elements of $F\setminus\{1\}$ to powers of themselves, but different from themselves by \cref{lem:Frobenius_General_Properties}. Let $h\in H$. If the generators $x$ and $y$ satisfy $hxh^{-1}=x^{a_1}$ and $hyh^{-1}=y^{a_2}$ with $1<a_2<a_1\leq p-1$, then $x^{a_2}y^{a_1}$ is conjugate to $(xy)^{a_2a_1}$, which contradicts the assumption on the fusion of $p$-subgroups. Hence, the necessary condition holds. 
The sufficient condition is clear since $G$ is a semi-direct product of $F$ by $H$ and $F$ is abelian.
Next, for any  $C\in X$ we have that  
$N_G(C)=G$ acts by conjugation on $C$. The induced group homomorphism
 \[
     \Theta: N_G(C)\longrightarrow \Aut(C)\cong C_{p-1}, g\mapsto c_g
 \]
 (where $c_g:C\longrightarrow C, c\mapsto gcg^{-1}$ is the automorphism of conjugation by $g$)  is such that $\Ker(\Theta)=F$ by \cref{lem:Frobenius_General_Properties}. This yields
 \[
       m = \frac{|G|}{|F|} \,\Big|\, |\Aut(C)|=p-1\,.
 \]

	\end{enumerate}	
\end{proof}

\begin{rem} 
Note that we excluded the prime number $2$. In this case, however, the situation is simple:  there is, up to isomorphism, only one Frobenius group of type $(C_2\times C_2)\rtimes C_m$, namely the alternating group $\mathfrak{A}_4=V_4\rtimes \langle(1\;2\;3)\rangle$, where $V_4$ is the Klein-four group.
Indeed, as we must have  $m\mid (p^2-1)=3$, the only possibility is $m=3$, and the only other non-abelian group of order~$12$ is the dihedral group of order~$12$, which does not have any normal subgroup isomorphic to $C_2\times C_2$. The trivial source character table $\Triv_2(\mathfrak{A}_4)$ can be found for example in \cite{BBthesis}, or in \cite{BBCLNF} through the isomorphism $\mathfrak{A}_4\cong \PSL_2(3)$.
\end{rem}     

\vspace{6mm}

\section{The maximal fusion case}
\label{sec:MaximalFusion}

We now turn to the computation of the trivial source character tables of metabelian Frobenius groups with Frobenius kernel $C_p\times C_p$ and cyclic Frobenius complement $C_{m}$,  where~$p$ is odd and there is precisely one conjugacy class of subgroups of order $p$.

\begin{nota}\label{nota:MaximalFusionCase}
Throughout this section, we assume that $G=F\rtimes H$ is a Frobenius group with Frobenius kernel $F\cong C_p\times C_p$ and cyclic Frobenius complement $H\cong C_{m}$,  where~$p$ is odd and $m=(p+1)\cdot\gcd(p-1,m)$. 
By~\cref{LemmaIsoMetabelianFrobeniusSpecialCases}, up to isomorphism, there is only one group of this type: it is a subgroup of $\text{AGL}_1(p^2)$ and there is precisely one conjugacy class of subgroups of order $p$ by \cref{prop:FrobeniusGroups_max_fusion_and_no_fusion}(a). 
\par
We set 
$d(m):=\gcd(p-1,m)$ and 
$e(m):=\frac{p^2-1}{m}=\frac{p-1}{d(m)}$, and notice that $2\mid d(m)$ by \cref{prop:FrobeniusGroups_max_fusion_and_no_fusion}(a). 
We let $H:=\langle h \rangle$ and $F:=\langle x \rangle\times \langle y \rangle$ and we choose the following set of representatives for the {$G$-conjugacy} classes of $p$-subgroups of $G$:
\begin{align*}
Q_1 &:= \TrivialGroup,\\
Q_2 &:= \langle x \rangle,\\
Q_3 &:= F.
\end{align*}
As in \cref{prop:decompositionMatrix_Frobenius_groups_all_cases_relevant_for_us}, we let $\Irr(G)=\{\chi_1^{}, \ldots,\chi_m^{},\chi_{m+1}^{},\ldots , \chi_{m+e(m)}^{} \}$ 
where for each $1\leq a\leq m$ 
we set~${\chi_a := \Inf_{G/F}^{G}(\xi_a)}$  with  $\xi_a\in\Irr(H)$  as defined in \cref{nota:ordinary_ct_of_cyclic_group}, and $\chi_{m+b} := \Ind_{F}^{G}(\nu_b)$  for each $1\leq b\leq e(m) $ and pairwise non-conjugate characters $\nu_b \in \Irr(F)\setminus \{1_{F}\}$.
\end{nota}

\begin{lem}\label{lem:Normalisers_AGL_1_p_squared}
With the notation introduced in~\cref{nota:MaximalFusionCase}, we have: 
\begin{enumerate}[label=\rm(\alph{enumi}),itemsep=1ex]
    \item $N_G(Q_1) = G$ and $\overline{N}_G(Q_1)\cong G$;
    \item $N_G(Q_2) = F \rtimes N_H(Q_2)$ 
    and $\overline{N}_G(Q_2) \cong\langle y \rangle \rtimes N_H(Q_2)$, where $|N_H(Q_2)|=d(m)$ and $N_H(Q_2)=\langle h^{p+1} \rangle$;
    \item $N_G(Q_3) = G$ and $\overline{N}_G(Q_3)\cong H$\,.
\end{enumerate}

\end{lem}

\begin{proof}
Assertions (a) and (c) are straightforward from the definitions. 
Assertion (b) follows from the fact that $G$ is a semi-direct product of $F$ by $H$ with $F$ abelian and \cref{prop:FrobeniusGroups_max_fusion_and_no_fusion}(a). 
\end{proof}

\begin{prop} \label{lem:Ordinary_ct_AGL_1_p_squared}\label{lem:Ordinary_ct_y_rtimes_h_p_plus_1}\label{prop:decompositionMatrixAGL_1_p_squared}\label{prop:decomp_matrix_N_2_bar}
Let $\zeta$ be a fixed primitive $m$-th root of unity in $K$ and let $\omega:=\zeta^{(p+1)}$. Then the following assertions hold.
\begin{enumerate}[label=\rm(\alph*),itemsep=1ex]
    \item

\noindent The ordinary character table of $G$ restricted to the $p$-regular conjugacy classes is as given in \cref{table:Ordinary_ct_AGL_1_p_squared}. 

\renewcommand*{\arraystretch}{1.1}
\begin{table}[h]
\begin{tabular}{|c||c|c|}
 \hline 
	&

		\textbf{$1$ } 

	&

	$h^j\ (1\leq j\leq m-1)$
	\\ \hline \hline

	$\chi_a\ (1\leq a\leq m)$
	& $1$
	& $\zeta^{(a-1)j}$
	\\ \hline
    
	$\chi_{m+b} (1\leq b\leq e(m))$
	& $m$
	& $0$
	\\ \hline	
\end{tabular}

    \caption{Ordinary character table of $G$ restricted to the $p$-regular classes.}
    \label{table:Ordinary_ct_AGL_1_p_squared}	
\end{table}

    \item 

     \noindent The ordinary character table of $\overline{N}_G(Q_2)$ restricted to the $p$-regular conjugacy classes is as given in \cref{table:Ordinary_ct_y_rtimes_h_p_plus_1}, 
     where following \cref{prop:decompositionMatrix_Frobenius_groups_all_cases_relevant_for_us} we let $\Irr(\overline{N}_G(Q_2))=\{\theta_1, \ldots , \theta_{d(m)+e(m)} \}$ where for each $1\leq b\leq e(m)$, $\theta_{b+d(m)} := \Ind_{\langle y \rangle}^{\overline{N}_G(Q_2)}(\nu_b)$ for pairwise non-conjugate characters $\nu_b \in \Irr(\langle y \rangle)\setminus \{1_{\langle y \rangle}\}$ and  for each $1\leq a\leq d(m)$ we let  $\theta_a := \Inf_{\overline{N}_G(Q_2)/\langle y\rangle}^{\overline{N}_G(Q_2)}(\xi_a)$  with  $\xi_a\in\Irr(N_H(Q_2))$  as defined in \cref{nota:ordinary_ct_of_cyclic_group}.

     \renewcommand*{\arraystretch}{1.1}
    \begin{table}[ht]
    \captionsetup{width=0.7\linewidth} 
	\begin{tabular}{|c||c|c|}
    \hline 
		&

			\textbf{$1$} 

		&

			$h^{j(p+1)}Q_2\ (1\leq j\leq d(m)-1)$
		\\ \hline \hline

		$\theta_{a}\ (1\leq a \leq d(m))$
		& $1$
		& $\omega^{(a-1)j}$
		\\ \hline

		$\theta_{d(m)+b}$ $(1\leq b\leq e(m))$
		& $d(m)$
		& $0$
		\\ \hline
		
	\end{tabular}
			\caption{Ordinary character table of $\overline{N}_G(Q_2)$ restricted to the $p$-regular conjugacy classes. }\label{table:Ordinary_ct_y_rtimes_h_p_plus_1}	
    \end{table}

    \item 

    Setting $\varphi_a:=\chi_a^{\circ}$ for each $1\leq a\leq m$, then $\IBr_p(G)=\{\varphi_1, \ldots, \varphi_{m}\}$ and $\decmat(G)$ is as given in \cref{table:decomp_matrix_AGL_1_p_squared}.

    \item 

    Setting $\psi_a:=\theta_a^{\circ}$ for each $1\leq a\leq d(m)$, then $\IBr_p(\overline{N}_G(Q_2))=\{\psi_1, \ldots, \psi_{d(m)}\}$ and  $\decmat(\overline{N}_G(Q_2))$ is as given in \cref{table:decomp_matrix_N_2_bar}.

\end{enumerate}

\par\bigskip

\makebox[0.8\textwidth][l]{%
  \parbox{0.45\textwidth}{%
    \centering
    \scalebox{0.95}{ 
\begin{NiceTabular}[nullify-dots,xdots/shorten=4pt]{c|ccccccc}
\hline
 & $\varphi_1$ & $\varphi_2$ & \Cdots & \Cdots & \Cdots & $\varphi_{m-1}$ & $\varphi_m$\\
\hline
$\chi_1$ & $1$ & $0$ & \Cdots & \Cdots & \Cdots  & $0$ & $0$ \\
$\chi_2$ & $0$ & $1$ & \Ddots & & & \Vdots & \Vdots \\
$\chi_3$ & $0$ & $0$ & \Ddots & \Ddots & & \Vdots & \Vdots \\
$\Vdots$ & $\Vdots$ & $\Vdots$ & $\Ddots$ & $\Ddots$ & $\Ddots$ & $\Vdots$ & $\Vdots$ \\
$\Vdots$ & $\Vdots$ & $\Vdots$ &  & $\Ddots$ & $\Ddots$ & $0$ & $0$ \\
$\Vdots$ & $\Vdots$ & $\Vdots$ &  &  & $\Ddots$ & $1$ & $0$ \\
$\chi_m$ & $0$ & $0$ & $\Cdots$ & $\Cdots$ & $\Cdots$ & $0$ & $1$\\
$\chi_{m+1}$ & $1$ & $1$ & $\Cdots$ & $\Cdots$ & $\Cdots$ & $1$ & $1$\\
$\Vdots$ & $\Vdots$ & \Vdots & & & & $\Vdots$ & $\Vdots$\\
$\chi_{m+e(m)}$ & $1$ & $1$ & $\Cdots$ & $\Cdots$ & $\Cdots$ & $1$ & $1$\\
\hline
\end{NiceTabular}
    }
    \captionof{table}{$p$-decomposition matrix of $G$.}
    \label{table:decomp_matrix_AGL_1_p_squared}
  }
  \hspace{0mm}
  \parbox{0.52\textwidth}{%
    \centering
    \scalebox{0.95}{
      \begin{NiceTabular}[nullify-dots,xdots/shorten=4pt]{c|ccccccc}
      \hline
      & $\psi_1$ & $\psi_2$ & \Cdots & \Cdots & \Cdots & $\psi_{d(m)-1}$ & $\psi_{d(m)}$\\
      \hline
      $\theta_1$ & $1$ & $0$ & \Cdots & \Cdots & \Cdots  & $0$ & $0$ \\
      $\theta_2$ & $0$ & $1$ & \Ddots & & & \Vdots & \Vdots \\
      $\theta_3$ & $0$ & $0$ & \Ddots & \Ddots & & \Vdots & \Vdots \\
      $\Vdots$ & $\Vdots$ & $\Vdots$ & \Ddots & \Ddots & \Ddots & \Vdots & \Vdots \\
      $\Vdots$ & $\Vdots$ & $\Vdots$ &  & \Ddots & \Ddots & $0$ & $0$ \\
      $\Vdots$ & $\Vdots$ & $\Vdots$ &  &  & \Ddots & $1$ & $0$ \\
      $\theta_{d(m)}$ & $0$ & $0$ & \Cdots & \Cdots & \Cdots & $0$ & $1$\\
       $\theta_{d(m)+1}$ & $1$ & $1$ & \Cdots & \Cdots & \Cdots & $1$ & $1$\\
      $\Vdots$ & $\Vdots$ & \Vdots & & & & $\Vdots$ & $\Vdots$\\
      $\theta_{d(m)+e(m)}$ & $1$ & $1$ & \Cdots & \Cdots & \Cdots & $1$ & $1$\\
      \hline
      \end{NiceTabular}
    }
    \captionof{table}{$p$-decomposition matrix of $\overline{N}_G(Q_2)$.}
    \label{table:decomp_matrix_N_2_bar}
  }
}

\end{prop}

\begin{proof}

\begin{enumerate}[label=\rm(\alph{enumi}),itemsep=1ex]
    \item 

For each $1\leq a\leq m$, we have $\chi_a = \Inf_{G/F}^{G}(\xi_a)$ with  $\xi_a\in\Irr(H)$. It follows from \cref{nota:ordinary_ct_of_cyclic_group} that $\chi_a(h^j) = \xi_{a}(h^j)= \zeta^{(a-1)j}$ for each $1\leq j \leq m-1$. Then, as~$\chi_{m+1},\ldots,\chi_{m+d(m)}$ are induced from  linear characters of $F$ to $G$ it is immediate that their degree is $m$ and they take value zero on the non-trivial elements of~$H$. 
    \item Analogous to (a) as $\overline{N}_G(Q_2)$ is also a Frobenius group. 
    \item and (d) are immediate from \cref{prop:decompositionMatrix_Frobenius_groups_all_cases_relevant_for_us}(d).
\end{enumerate}

\end{proof}

We now come to our first main result, namely the description of the trivial source character table of~$G$ at $p$. In this case, it follows from the definition, that this is a $3\times 3$-block matrix.

\begin{thm}\label{Thm_Triv_p_AGL_1_p_squared}
		Assume $p$ is an odd prime number and $G=F\rtimes H$ is a Frobenius group with $F\cong C_p\times C_p$ and $H\cong C_m$ where $m$ is an integer such that $m=(p+1)\cdot\mathrm{gcd}(m,p-1)$.
        Then,
        with the notation introduced in \cref{nota:MaximalFusionCase} and \cref{lem:Ordinary_ct_AGL_1_p_squared}, 
        the trivial source character table $\Triv_p(G)=[T_{i,v}]_{1\leq i,v\leq 3}$ of $G$ is given as described below.
		\begin{enumerate}[label={\rm(\alph*)}]
        \item The labelling of the columns may be chosen as follows:
                \begin{enumerate}[label={\rm(\arabic*)}]
                   \item 
                   the columns of $T_{i,1}$ $(1\leq i\leq 3)$ may be  labelled by the elements of $H$;
                   \item 
                   the columns of $T_{i,2}$ $(1\leq i\leq 3)$ may be  labelled by the elements of $N_H(Q_2)$;
                   \item 
                   the columns of $T_{i,3}$ $(1\leq i\leq 3)$ may be  labelled by the elements of $H$. 
               \end{enumerate}
            \item The ordinary characters of the trivial source modules are as follows:
               \begin{enumerate}[label={\rm(\arabic*)}]
                   \item
                   $\{\chi^{}_{\widehat{M}}\mid  M\in \TS(G;Q_1)\}=\{\chi+\sum_{\psi\in\Irr(G)\setminus\Lin(G)}\psi\mid \chi\in\Lin(G)\}$;
                   \item 
                   $\{\chi^{}_{\widehat{M}}\mid  M\in \TS(G;Q_2)\}
                   =
                   \{\,
                   \sum\limits_{\chi\in\Irr(G)\setminus\Lin(G)}^{}
                   \chi\,
                   +
                   \sum\limits_{\substack{\lambda\in\Lin(G),\\
                   \Res^G_{N_H(Q_2)}(\lambda)=\theta}}^{}\lambda
                   \mid \theta\in \Irr(N_H(Q_2))
                   \}
                   $;
                   \item 
                   $\{\chi^{}_{\widehat{M}}\mid  M\in \TS(G;Q_3)\}=\Lin(G)$. 
               \end{enumerate}
               As $\Lin(G)=\Inf_{H}^G(\Irr(H))$ we may choose the labelling of the rows of $T_{1,v}$ and $T_{3,v}$ ${(1\leq v\leq 3)}$ to match that of $X(H)$. Similarly,  we may choose the labelling of the rows of $T_{2,v}$ $(1\leq v\leq 3)$ to match that of $X(N_H(Q_2))$.
			\item  
            With the labelling of the rows and of the columns given in {\rm(a)} and {\rm(b)} we have:
            \begin{enumerate}[label={\rm(\arabic*)}]
                \item 
            $T_{1,2}=T_{1,3}=T_{2,3}={\mathbf{0}}$\,;
   			\item $T_{1,1}=X(H) + \left(\begin{smallmatrix}
                p^2-1 & 0 & \cdots & 0\\
                \vdots & \vdots & \ddots & \vdots\\
                p^2-1 & 0 & \cdots & 0 \end{smallmatrix}\right)$;
             \item $T_{2,2}=X(N_H(Q_2)) +\left(\begin{smallmatrix} 
             p-1 & 0 & \cdots & 0\\
             \vdots & \vdots & \ddots & \vdots\\
             p-1 & 0 & \cdots & 0 \end{smallmatrix}\right)$;
            \item $T_{2,1}=(A_{M,s})_{M\in\TS(G;Q_2), s\in H}$  where $A_{M,s}=0$ provided  $s\in H\setminus N_H(Q_2)$ and  
            $$(A_{M,s})_{M\in\TS(G;Q_2), s\in N_H(Q_2)}=(p+1)\cdot T_{2,2}\,;$$
            \item  $T_{3,2}=(B_{M,s})_{M\in\TS(G;Q_3), s\in N_H(Q_2) }$  where
              $B_{M,s} = (T_{1,3})_{M,s}$ for each $M\in\TS(G;Q_3)$ and each $s\in N_H(Q_2)$;  
            \item
            $T_{3,1}=T_{3,3} = X(H)$.
            \end{enumerate}
		\end{enumerate}
  \end{thm}

\noindent To help visualize this result, the table version of \cref{Thm_Triv_p_AGL_1_p_squared} for $G=\AGL_1(p^2)$ is given in Appendix~\ref{app:tables}.

\begin{proof}
Assertion (a) is straightforward from \cref{conv:tsctbl}, \cref{prop:decompositionMatrix_Frobenius_groups_all_cases_relevant_for_us}(c) and  \cref{lem:Normalisers_AGL_1_p_squared}.
Next, we prove Assertion (b).  
To simplify, in this proof we set $N_2:=N_G(Q_2)$ and $\overline{N}_2:=\overline{N}_G(Q_2)$. 
\begin{enumerate}[label=\rm(\arabic*), leftmargin=8mm]
    \item The ordinary characters of the trivial source modules in $\TS(G;Q_1)$ are the ordinary characters of the PIMs of $kG$ and can be read off from the decomposition matrix in \cref{table:decomp_matrix_AGL_1_p_squared}. The claim follows. 
    \item

    Let $M\in\TS(G;Q_2)$. By the bijections in Subsection~\ref{ssec:DefTSCT} there exists a unique PIM $P_{\psi_a}$~of $k\overline{N}_2$ with $1\leq a\leq d(m)$     such that $M$ is the Green correspondent of the inflated module $\Inf_{N_2/\langle x\rangle}^{N_2} (P_{\psi_a})$. By \cref{Proposition1AbelianGroupsTSModules}, the induced module
$$\Ind_{N_2}^{G} \Inf_{N_2/\langle x \rangle}^{N_2} (P_{\psi_a})$$
is indecomposable,  
so it must be $M$. Next, we compute the constituents of the ordinary character of this module. First, by the above, 
\[
    \chi^{}_{\widehat{M}}
    =
    \Ind_{N_2}^{G} \Inf_{N_2/\langle x \rangle}^{N_2}  (\Phi_{\psi_a}),
\]
and by \cref{prop:decomp_matrix_N_2_bar}(d) the ordinary character of $P_{\psi_a}$ is
$$\Phi_{\psi_a} = \theta_a + \sum\limits_{b=1}^{e(m)} \theta_{d(m)+b}\,.$$
Now, we claim that there is a bijection
\begin{center}
\begin{tabular}{cccl}
       $\Ind_{N_2}^{G}\Inf_{\overline{N}_2}^{N_2}$\,:            &   $\Irr(\overline{N}_2)\setminus\Lin(\overline{N}_2)$      & $\lra$ &    $\Irr(G)\setminus\Lin(G)$.   \\
\end{tabular}
\end{center}

\noindent Given $u\in \{1, \ldots , e(m)\}$, as $N_2$ is a Frobenius group, we can write $\Inf_{\overline{N}_2}^{N_2} (\theta_{d(m) + u}) = \Ind_F^{N_2}(\nu_{u}) =: I_u$ for some character $\nu_u\in \Irr(F)\setminus \{1_F\}$ by \cref{prop:Ordinary_Irreducible_Character_Of_Frobenius_Groups}. As $G$ is also a Frobenius group, by transitivity of induction, we obtain that $\Ind_{N_2}^{G} \Inf_{\overline{N}_2}^{N_2} (\theta_{d(m)+u})$ is irreducible. Hence, the map is well-defined. As both sets have the same cardinality, it remains to show that the map is injective. Assuming $\theta_{d(m)+u}\neq \theta_{d(m)+w}$ for $u,w\in \{1, \ldots , e(m)\}$, then $I_u\neq I_w$. By the above, $\Ind_{N_2}^G(I_u)=: \chi_u\in \Irr(G)\setminus \Lin(G)$. Then, by Clifford theory, the inertial subgroup of $I_u$ is $N_2$  and 
$\Res_{N_2}^G(\chi_u) = \sum\limits_{g\in [G/N_2]}^{}{I_u}^{(g)}$.
Moreover, we can choose the non-trivial coset representatives $g\in [G/N_2]$ in such a way that $g=:\ell\in H$.
Now, we show that  $I_w$ is not a constituent of $\Res_{N_2}^G(\chi_u)$. As $\ell\notin N_2$ there exists $q\in Q_2 \setminus \{1\}$ such that $\ell q{\ell}^{-1} \in F\setminus Q_2$. Since both $I_u$ and $I_w$ are induced from $F$, they do not have $F$ in their kernels. Hence, ${I_u}^{(\ell)}=I_w$ is not possible for such an $\ell$, as $\ker(I_u)\cap F =Q_2 = \ker(I_w)\cap F$. Hence, $\Ind_{N_2}^G(I_u)\neq \Ind_{N_2}^G(I_w)$ and  the map is injective.\\
It follows that 
$$\Ind_{N_G(Q_2)}^{G}\Inf_{\overline{N}_2}^{N_2}(\sum\limits_{b=1}^{e(m)} \theta_{d(m)+b}) = \sum\limits_{\chi\in\Irr(G)\setminus \Lin(G)}^{} {\chi}\,.$$
Now we compute $\Ind_{N_2}^{G}\Inf_{\overline{N}_2}^{N_2} (\theta_a)$. 
Clearly, the degree of this character is        
    \[
    |G:N_2|=|H:N_H(Q_2)|=\frac{m}{d(m)}=p+1<m\,,
    \]
so by \cref{lem:Ordinary_ct_AGL_1_p_squared}(a), it must be a sum of $p+1$ linear characters of $G$. Now,  Clifford's theorem together with  Gallagher's theorem  (see \cite[19.3 Theorem and 19.5 Theorem]{HuppertCharacterTheory}) tell us that we can write
    \[
    \Ind_{N_2}^{G}\Inf_{\overline{N}_2}^{N_2} (\theta_a)=\sum_{c=1}^{p+1}\lambda_c
    \]
with  $\lambda_1,\ldots,\lambda_{p+1}\in\Lin(G)$ pairwise distinct and 
    \[
    \Res^G_{N_2}\left(\Ind_{N_2}^{G}\Inf_{\overline{N}_2}^{N_2} (\theta_a)\right)
    =
    (p+1)\cdot \Inf_{\overline{N}_2}^{N_2} (\theta_a)\,,
    \]
implying that $\Res^G_{N_2}(\lambda_c)=\Inf_{\overline{N}_2}^{N_2} (\theta_a)$ for every $1\leq c\leq p+1$. 
Moreover, identifying $N_2/F$ with $N_H(Q_2)$, then  \cite[\S1.1.3(2.e.)]{BoucBisetBook} yields $\Res^{N_2}_{N_H(Q_2)}\left(\Inf_{\overline{N}_2}^{N_2} (\theta_a)\right)=\theta_a$. 
A counting argument shows that any linear character of $G$ must be a constituent of an induced character $\Ind_{N_2}^{G}\Inf_{\overline{N}_2}^{N_2} (\theta_{a_0})$ for some index $1\leq a_0\leq d(m)$. Thus, we have proved that  
    \[
        \Ind_{N_2}^{G}\Inf_{\overline{N}_2}^{N_2} (\theta_a)=\sum_{c=1}^{p+1}\lambda_c = \sum\limits_{\substack{\lambda\in\Lin(G),\\
                   \Res^G_{N_H(Q_2)}(\lambda)=\theta_a}}^{}\lambda\,, 
    \]
as required.

    \item \cref{lem:Normalisers_AGL_1_p_squared}(c) and \cref{lem:Simples_G_mod_Op(G)}(b) yield
    $$\{\chi^{}_{\widehat{M}}\mid M\in\TS(G;Q_3)\}=\Inf_{G/Q_3}^{G}(\Irr(G/Q_3))\subseteq \Lin(G)$$
    and the last inclusion is an equality by \cref{lem:Frobenius_General_Properties}(d) as $|\Lin(G)|=|G/[G,G]|$. 
\end{enumerate}
\noindent Finally, we prove Assertion (c). 
\begin{enumerate}[label=\rm(\arabic*), leftmargin=8mm]
    \item The claim is immediate from \cref{lem:PropTrivpG}(a).
    \item By \cref{lem:PropTrivpG}(b), we have $$T_{1,1}= \Phi_p(G)= \decmat(G)^{t}\cdot X(G,p')\,.$$
    Using 
    \cref{lem:Ordinary_ct_AGL_1_p_squared}(a) and (c) and noticing that with our choice of the labelling for the rows and the columns, 
    $X(G,p')$ is a block matrix of type
        \[
        \begin{pNiceArray}{ccw{c}{1cm}c}[margin]\Block{2-4}<\large>{X(H)} & & & \\& & &  \\
        \hline
        m & 0 &  \Cdots & 0 \\
        \vdots & \vdots & & \vdots\\
        m & 0 &  \Cdots & 0
        \end{pNiceArray} 
        \in K^{(m+e(m)) \times m},
        \]
    the product is easily calculated to be as claimed, since $m\cdot e(m)=p^2-1$. 
    \item Similarly, by \cref{lem:PropTrivpG}(b), we have $$T_{2,2}= \Phi_p(\overline{N}_2)= \decmat(\overline{N}_2)^{t}\cdot X(\overline{N}_2,p').$$
    This product is easily computed from \cref{lem:Ordinary_ct_AGL_1_p_squared}(b) and (d), noticing that with our choice of the labelling for the rows and the columns, 
    $X(\overline{N}_2,p')$ is a block matrix of type
    \[
    \begin{pNiceArray}{ccw{c}{1cm}c}[margin]
        \Block{2-4}<\large>{X(N_H(Q_2))} & & & \\
        & & &  \\
        \hline
        d(m) & 0 &  \Cdots & 0 \\
        \vdots & \vdots & & \vdots\\
        d(m) & 0 &  \Cdots & 0 
    \end{pNiceArray}
    \in K^{(d(m)+e(m)) \times d(m)}
    \] 
    and $d(m)\cdot e(m)=p-1$. 
    \item By \cref{lem:PropTrivpG}(c), the entries of $T_{2,1}$  
    are obtained by evaluating the ordinary characters of the trivial source modules with vertex $Q_2$ obtained in (b)(2) at the elements of $H$. 
    As the latter characters are induced from $N_2=F\rtimes N_H(Q_2)$, it is clear that they take value zero on $H\setminus N_H(Q_2)$. In addition, it is also clear from Part (b)(2) (and its proof) that the values of these characters at the elements of $N_H(Q_2)$ are the values of the characters of the corresponding PIMs of  $\overline{N}_2$ multiplied by the index $|G:N_2|=p+1$\,.
    \item By definition, we have $T_{3,2} = \big[ \tau_{Q_{2},s}^{G}([M])\big]_{M\in \TS(G;Q_{3}), s\in [\overline{N}_G(Q_2)]_{p^\prime} }\,.$
     So, let $M\in\TS(G;Q_3)$, which is the set of all $1$-dimensional $kG$-modules.  By \cref{lem:Normalisers_AGL_1_p_squared} and \cref{prop:decompositionMatrix_Frobenius_groups_all_cases_relevant_for_us}(c), we may assume that $[N_2]_{p^\prime}= N_H(Q_2)$ and let $t\in [N_2]_{p^\prime}$. By definition, 
     $$\tau_{Q_2, t}^G([M]) = \varphi_{M[Q_2]}(t).$$ 
     Because $Q_2\leq Q_3=F \unlhd G$, it follows from  \cite[Proposition 5.10.4]{Linckelmann1} that $M[Q_2]=\Res^G_{N_2}(M)$. Hence, $\varphi_{M[Q_2]}(t) = \chi_{\widehat{M}}(t)$, proving that the values of $T_{3,2}$ are obtained from $T_{3,1}$ by restriction to the columns labelled by the elements of $N_H(Q_2)$. 
    \item On the one hand, since $\overline{N}_G(Q_3)\cong H$,  which is a $p^\prime$-group, by \cref{lem:PropTrivpG}(b), we have 
    $$T_{3,3}= \Phi_p(H) = \decmat(H)^{t}\cdot X(H,p') = X(H),$$ as claimed.
    On the other hand, by \cref{lem:PropTrivpG}(c) we have 
$$T_{3,1}=\big( \chi^{}_{\widehat{M}}(s)\big)_{M\in \TS(G;Q_{3}), s\in [G]_{p^\prime} }.$$
 Therefore, it follows from Part (a)(3) and Part (b)(3) that $T_{3,1}=X(H)$.  
\end{enumerate}  

\end{proof}

\begin{exmp}
Let $G$ be the Frobenius group $(C_3\times C_3)\rtimes C_{8}$ of order $72$ with maximal fusion pattern, i.e. $G$ has precisely one conjugacy class of subgroups of order $3$. 
It follows that we have three conjugacy classes of $3$-subgroups of $G$, namely
$$Q_1=\{1\},\quad Q_2\cong C_3,\quad  Q_3\cong C_3\times C_3.$$
Notice that $G$ is isomorphic to the group labelled by\ [ 72, 39 ]\ in GAP's SmallGroups library, see~\cite{gap}. The ordinary character table of $G$ is as given in \cref{Ordinary_ct_FrobGroup_C3_C3_C8}, where $\zeta_8 := \exp(\frac{2\pi i}{8})$.
\begin{center}
\begin{table}[h]
\scalebox{0.84}{
\begin{tabular}{@{}l@{}l@{}l@{}}
\hline
\(\begin{array}{|l|ccccccccc|}
  & 1a & 8a & 4a & 8b & 2a & 8c & 4b & 8d & 3a\\ \hline
\chi_{1} & 1 & 1 & 1 & 1 & 1 & 1 & 1 & 1 & 1\\
\chi_{2} & 1 & \zeta_8 & \zeta_8^2 & \zeta_8^{3} & -1  & -\zeta_8 & -\zeta_8^2 & -\zeta_8^{3} & 1\\
\chi_{3} & 1 & \zeta_8^2 & -1 & -\zeta_8^2 & 1 & \zeta_8^2 & -1 & -\zeta_8^2 & 1\\
\chi_{4} & 1 & \zeta_8^{3} & -\zeta_8^2 & \zeta_8 & -1 & -\zeta_8^{3} & \zeta_8^2 & -\zeta_8 & 1\\
\chi_{5} & 1 & -1 & 1 & -1 & 1 & -1 & 1 & -1 & 1\\
\chi_{6} & 1 & -\zeta_8 & \zeta_8^2 & -\zeta_8^{3} & -1 & \zeta_8 & -\zeta_8^2 & \zeta_8^{3} & 1\\
\chi_{7} & 1 & -\zeta_8^2 & -1 & \zeta_8^2 & 1 & -\zeta_8^2 & -1 & \zeta_8^2 & 1\\
\chi_{8} & 1 & -\zeta_8^{3} & -\zeta_8^2 & -\zeta_8 & -1 & \zeta_8^{3} & \zeta_8^2 & \zeta_8 & 1\\
\chi_{9} & 8 & 0 & 0 & 0 & 0 & 0 & 0 & 0 & -1\\
\hline
\end{array}\)\\
\end{tabular}
}
\caption{Ordinary character table of $ (C_3\times C_3)\rtimes C_8$.}
\title{}
\label{Ordinary_ct_FrobGroup_C3_C3_C8}
\end{table}

\end{center}
The trivial source character table $\Triv_3(G)$ is as given in \cref{TSCT_FrobGrp_72}. Following our conventions we label the columns of $\Triv_3(G)$ with $3^\prime$-elements in $N_v$ instead of $\overline{N}_v$\ $(1\leq v\leq 3)$. 

\begin{center}
\begin{table}[h]
\scalebox{0.66}{
\begin{tabular}{@{}l@{}l@{}l@{}l@{}l@{}l@{}l@{}l@{}l@{}l@{}}
\(\begin{array}{|l|cccccccc|cc|cccccccc|}
\hline
\textup{Normalisers}\ N_i & \multicolumn{8}{c|}{N_{1}\cong (C_3\times C_3)\rtimes C_8} & \multicolumn{2}{c|}{N_{2}\cong (C_3\times C_3)\rtimes C_2} & \multicolumn{8}{c|}{N_{3}\cong (C_3\times C_3)\rtimes C_8}\\ \hline
\textup{Representatives}\ n_j\ \in\ N_i & 1a & 8a & 4a & 8b & 2a & 8c & 4b & 8d & 1a & 2a & 1a & 8a & 4a & 8b & 2a & 8c & 4b & 8d\\ \hline
\chi_{1} + \chi_{9} & 9  & 1 & 1 & 1 & 1 & 1 & 1 & 1 & 0 & 0 & 0 & 0 & 0 & 0 & 0 & 0 & 0 & 0\\
\chi_{2} + \chi_{9} & 9 & \zeta_8 & \zeta_8^2 & \zeta_8^{3} & -1  & -\zeta_8 & -\zeta_8^2 & -\zeta_8^{3} & 0 & 0 & 0 & 0 & 0 & 0 & 0 & 0 & 0 & 0\\
\chi_{3} + \chi_{9} & 9 & \zeta_8^2 & -1 & -\zeta_8^2 & 1 & \zeta_8^2 & -1 & -\zeta_8^2 & 0 & 0 & 0 & 0 & 0 & 0 & 0 & 0 & 0 & 0\\
\chi_{4} + \chi_{9} & 9 & \zeta_8^{3} & -\zeta_8^2 & \zeta_8 & -1 & -\zeta_8^{3} & \zeta_8^2 & -\zeta_8 & 0 & 0 & 0 & 0 & 0 & 0 & 0 & 0 & 0 & 0\\
\chi_{5} + \chi_{9} & 9 & -1 & 1 & -1 & 1 & -1 & 1 & -1 & 0 & 0 & 0 & 0 & 0 & 0 & 0 & 0 & 0 & 0\\
\chi_{6} + \chi_{9} & 9 & -\zeta_8 & \zeta_8^2 & -\zeta_8^{3} & -1 & \zeta_8 & -\zeta_8^2 & \zeta_8^{3} & 0 & 0 & 0 & 0 & 0 & 0 & 0 & 0 & 0 & 0\\
\chi_{7} + \chi_{9} & 9 & -\zeta_8^2 & -1 & \zeta_8^2 & 1 & -\zeta_8^2 & -1 & \zeta_8^2 & 0 & 0 & 0 & 0 & 0 & 0 & 0 & 0 & 0 & 0\\
\chi_{8} + \chi_{9} & 9 & -\zeta_8^{3} & -\zeta_8^2 & -\zeta_8 & -1 & \zeta_8^{3} & \zeta_8^2 & \zeta_8 & 0 & 0 & 0 & 0 & 0 & 0 & 0 & 0 & 0 & 0\\
 \hline
\chi_{1} + \chi_{3} + \chi_{5} + \chi_{7} + \chi_{9} & 12 & 0 & 0 & 0 & 4 & 0 & 0 & 0 & 3 & 1 & 0 & 0 & 0 & 0 & 0 & 0 & 0 & 0\\
\chi_{2} + \chi_{4} + \chi_{6} + \chi_{8} + \chi_{9} & 12 & 0 & 0 & 0 & -4 & 0 & 0 & 0 & 3 & -1 & 0 & 0 & 0 & 0 & 0 & 0 & 0 & 0\\
 \hline
\chi_{1} & 1 & 1 & 1 & 1 & 1 & 1 & 1 & 1 & 1 & 1 & 1 & 1 & 1 & 1 & 1 & 1 & 1 & 1\\
\chi_{2} & 1 & \zeta_8 & \zeta_8^2 & \zeta_8^{3} & -1  & -\zeta_8 & -\zeta_8^2 & -\zeta_8^{3} & 1 & -1 & 1 & \zeta_8 & \zeta_8^2 & \zeta_8^{3} & -1  & -\zeta_8 & -\zeta_8^2 & -\zeta_8^{3}\\
\chi_{3} & 1 & \zeta_8^2 & -1 & -\zeta_8^2 & 1 & \zeta_8^2 & -1 & -\zeta_8^2 & 1 & 1 & 1 & \zeta_8^2 & -1 & -\zeta_8^2 & 1 & \zeta_8^2 & -1 & -\zeta_8^2\\
\chi_{4} & 1 & \zeta_8^{3} & -\zeta_8^2 & \zeta_8 & -1 & -\zeta_8^{3} & \zeta_8^2 & -\zeta_8 & 1 & -1 & 1 & \zeta_8^{3} & -\zeta_8^2 & \zeta_8 & -1 & -\zeta_8^{3} & \zeta_8^2 & -\zeta_8\\
\chi_{5} & 1 & -1 & 1 & -1 & 1 & -1 & 1 & -1 & 1 & 1 & 1 & -1 & 1 & -1 & 1 & -1 & 1 & -1\\
\chi_{6} & 1 & -\zeta_8 & \zeta_8^2 & -\zeta_8^{3} & -1 & \zeta_8 & -\zeta_8^2 & \zeta_8^{3} & 1 & -1 & 1 & -\zeta_8 & \zeta_8^2 & -\zeta_8^{3} & -1 & \zeta_8 & -\zeta_8^2 & \zeta_8^{3}\\
\chi_{7} & 1 & -\zeta_8^2 & -1 & \zeta_8^2 & 1 & -\zeta_8^2 & -1 & \zeta_8^2 & 1 & 1 & 1 & -\zeta_8^2 & -1 & \zeta_8^2 & 1 & -\zeta_8^2 & -1 & \zeta_8^2\\
\chi_{8} & 1 & -\zeta_8^{3} & -\zeta_8^2 & -\zeta_8 & -1 & \zeta_8^{3} & \zeta_8^2 & \zeta_8 & 1 & -1 & 1 & -\zeta_8^{3} & -\zeta_8^2 & -\zeta_8 & -1 & \zeta_8^{3} & \zeta_8^2 & \zeta_8\\
\hline
\end{array}\)\end{tabular}
}
\caption{Trivial source character table of $(C_3\times C_3)\rtimes C_8$ at $p=3$.}
\label{TSCT_FrobGrp_72}
\title{}
\end{table}
\end{center}
\end{exmp}


\vspace{8mm}

\section{The minimal fusion case}
\label{sec:NO_FUSION}

We now turn to the computation of the trivial source character tables of metabelian Frobenius groups with Frobenius kernel $F\cong C_p\times C_p$ and cyclic Frobenius complement $H\cong C_{m}$ such that no two distinct $p$-subgroups of~$G$ of order $p$ are $G$-conjugate. Recall from \cref{prop:FrobeniusGroups_max_fusion_and_no_fusion}(b) that $|H|$ divides $p-1$ and that $H$ acts on $F$ by raising each element to a power of itself in this case. Moreover, we have $m>1$.

\begin{nota}\label{nota:MinimalFusionCase}
Throughout this section, we adopt the following notation.
We choose a generator~$h$ of $C_{p-1}$ and let $H:=\langle h^{a(m)} \rangle$, where $a(m):=\frac{p-1}{m}$. We let $\{x,y\}$ be a set of generators for~$F$. By~\cref{prop:FrobeniusGroups_max_fusion_and_no_fusion}, we can choose the following set of representatives for the $G$-conjugacy classes of $p$-subgroups of $G$:
\begin{align*}
Q_1 &:= \TrivialGroup,\\
Q_i &:= \langle x\cdot y^{i-2}\rangle \quad (2\leq i\leq p+1),\\
Q_{p+2} &:= \langle y\rangle,\\
Q_{p+3} &:= F.
\end{align*}
\noindent Then, up to $G$-conjugation, the lattice of subgroups of $G$ of order $p$ is as given below.

	$$\xymatrix{ & & & Q_{p+3} & & & \\
		Q_2\ar@{-}[rrru] & Q_3\ar@{-}[rru] & Q_4\ar@{-}[ru] & \cdots & Q_p\ar@{-}[lu] & Q_{p+1}\ar@{-}[llu] & Q_{p+2}\ar@{-}[lllu] \\
		 & & & Q_1 \ar@{-}[lllu] \ar@{-}[llu] \ar@{-}[lu] \ar@{-}[ru] \ar@{-}[rru] \ar@{-}[rrru] & & & }$$
\noindent As in \cref{prop:decompositionMatrix_Frobenius_groups_all_cases_relevant_for_us}, we let $\Irr(G)=\{\chi_1, \ldots , \chi_{m+(p+1)\cdot a(m)}\}$ where for each $1\leq a\leq m$ we set~$\chi_a := \Inf_{G/F}^{G}(\xi_a)$  with  $\xi_a\in\Irr(H)=\Lin(H)$  as defined in \cref{nota:ordinary_ct_of_cyclic_group}, and 
$$\{\chi_{m+1},\ldots, \chi_{m+(p+1)\cdot a(m)}\}  = \{ \Ind_{F}^{G}(\nu)\ |\ \nu\in T\}$$
where $T$ is a set of representatives for the conjugation action of $G$ on $\Irr(F)\setminus \{1_{F}\}$\,.
\end{nota}

\begin{lem}\label{lem:Normalisers_MinimalFusionCase}
With the notation introduced in \cref{nota:MinimalFusionCase} the following assertions hold:
\begin{enumerate}[label=\rm(\alph{enumi}),itemsep=1ex]
\item $N_G(Q_v)=G$ and $\overline{N}_G(Q_v) = G/Q_v$ for every $1\leq v\leq p+3$; and
\item $\overline{N}_G(Q_v)$ is a Frobenius group with Frobenius complement $H\cong HQ_v/Q_v$ and Frobenius kernel $F/Q_v$ for every $1\leq v\leq p+2$.
\end{enumerate}
\end{lem}

\begin{proof}
\begin{enumerate}[label=\rm(\alph{enumi}),itemsep=1ex]
\item As no two distinct $p$-subgroups of~$G$ of order $p$ are $G$-conjugate, it follows that $N_G(Q_v)=G$ for each $1\leq v\leq p+3$. The second claim is then immediate.
\item For each $1\leq v\leq p+2$, clearly $Q_v\propernormalsubgroup F$ and $Q_v\propernormalsubgroup G$  by (a). As $G$ is a Frobenius group with respect to $H$,  the assertion follows from the definition.\qedhere
\end{enumerate}
\end{proof}

\begin{nota}\label{nota:Irr_N_bar_minimal_fusion}
    Following \cref{prop:decompositionMatrix_Frobenius_groups_all_cases_relevant_for_us}, given $v\in\{2,\ldots, p+2\}$, we let $\Irr(\overline{N}_G(Q_v))=$ $\{\theta_1, \ldots , \theta_{m+\frac{p-1}{m}}\}$ where for each $1\leq a\leq m$ we set~$\theta_a := \Inf_{(G/Q_v)/(F/Q_v)}^{G/Q_v}(\xi_a)$  with  $\xi_a\in\Irr(H)$  as defined in \cref{nota:ordinary_ct_of_cyclic_group}, and 
$$\{\theta_{m+1},\ldots, \theta_{m+\frac{p-1}{m}}\}  = \{ \Ind_{F/Q_v}^{G/Q_v}(\nu)\ |\ \nu\in V\}$$
where $V$ is a set of representatives for the conjugation action of $G/Q_v$ on $\Irr(F/Q_v)\setminus \{1_{F/Q_v}\}$\,.
\end{nota}

\begin{lem}\label{prop:decompositionMatrix_G_NO_FUSION}
\begin{enumerate}[label=\rm(\alph{enumi}),itemsep=1ex]
    \item Setting $\varphi_a:=\chi_a^{\circ}$ for each $1\leq a\leq m$, we have  $\IBr_p(G)=\{\varphi_1, \ldots, \varphi_{m}\}$ and $\decmat(G)$ is as given in \cref{table:decompositionMatrix_G_NO_FUSION}.
    
    \item Setting $\psi_a:=\theta_a^{\circ}$ for each $1\leq a\leq m$, we have  $\IBr_p(\overline{N}_G(Q_v))=\{\psi_1, \ldots, \psi_{m}\}$ and $\decmat(\overline{N}_G(Q_v))$ is as given in \cref{table:decompositionMatrix_N_bar_NO_FUSION} for every  $2\leq v\leq p+2$\,.

\end{enumerate}
\end{lem}

\par\bigskip\bigskip
\makebox[0.95\textwidth][l]{%
  \parbox{0.48\textwidth}{%
    \centering
    \scalebox{0.85}{ 
      \begin{NiceTabular}[nullify-dots,xdots/shorten=4pt]{c|ccccccc}
      \hline
       & $\varphi_1$ & $\varphi_2$ & \Cdots & \Cdots & \Cdots & $\varphi_{m-1}$ & $\varphi_m$\\
      \hline
      $\chi_1$ & $1$ & $0$ & \Cdots & \Cdots & \Cdots  & $0$ & $0$ \\
      $\chi_2$ & $0$ & $1$ & \Ddots & & & \Vdots & \Vdots \\
      $\chi_3$ & $0$ & $0$ & \Ddots & \Ddots & & \Vdots & \Vdots \\
      $\Vdots$ & $\Vdots$ & $\Vdots$ & \Ddots & \Ddots & \Ddots & \Vdots & \Vdots \\
      $\Vdots$ & $\Vdots$ & $\Vdots$ &  & \Ddots & \Ddots & $0$ & $0$ \\
      $\Vdots$ & $\Vdots$ & $\Vdots$ &  &  & \Ddots & $1$ & $0$ \\
      $\chi_m$ & $0$ & $0$ & \Cdots & \Cdots & \Cdots & $0$ & $1$\\
      $\chi_{m+1}$ & $1$ & $1$ & \Cdots & \Cdots & \Cdots & $1$ & $1$\\
      $\Vdots$ & $\Vdots$ & \Vdots & & & & \Vdots & \Vdots\\
      $\chi_{m+(p+1)\cdot a(m)}$ & $1$ & $1$ & \Cdots & \Cdots & \Cdots & $1$ & $1$\\
      \hline
      \end{NiceTabular}
    }
    \captionof{table}{$p$-decomposition matrix of $G$.\\ \phantom{22}}
    \label{table:decompositionMatrix_G_NO_FUSION}
  }
  \hspace{0.2cm}
  \parbox{0.48\textwidth}{%
    \centering
    \scalebox{0.85}{
      \begin{NiceTabular}[nullify-dots,xdots/shorten=4pt]{c|ccccccc}
      \hline
       & $\psi_1$ & $\psi_2$ & \Cdots & \Cdots & \Cdots & $\psi_{m-1}$ & $\psi_m$\\
      \hline
      $\theta_1$ & $1$ & $0$ & \Cdots & \Cdots & \Cdots  & $0$ & $0$ \\
      $\theta_2$ & $0$ & $1$ & \Ddots & & & \Vdots & \Vdots \\
      $\theta_3$ & $0$ & $0$ & \Ddots & \Ddots & & \Vdots & \Vdots \\
      $\Vdots$ & $\Vdots$ & $\Vdots$ & \Ddots & \Ddots & \Ddots & \Vdots & \Vdots \\
      $\Vdots$ & $\Vdots$ & $\Vdots$ &  & \Ddots & \Ddots & $0$ & $0$ \\
      $\Vdots$ & $\Vdots$ & $\Vdots$ &  &  & \Ddots & $1$ & $0$ \\
      $\theta_m$ & $0$ & $0$ & \Cdots & \Cdots & \Cdots & $0$ & $1$\\
      $\theta_{m+1}$ & $1$ & $1$ & \Cdots & \Cdots & \Cdots & $1$ & $1$\\
      $\Vdots$ & $\Vdots$ & \Vdots & & & & \Vdots & \Vdots\\
      $\theta_{m+\frac{p-1}{m}}$ & $1$ & $1$ & \Cdots & \Cdots & \Cdots & $1$ & $1$\\
      \hline
      \end{NiceTabular}
    }
    \captionof{table}{$p$-decomposition matrix of $\overline{N}_G(Q_v)$\\($2\leq v\leq p+2$).}
    \label{table:decompositionMatrix_N_bar_NO_FUSION}
  }
}

\begin{proof}
Both (a) and (b) are immediate from \cref{prop:decompositionMatrix_Frobenius_groups_all_cases_relevant_for_us}(d).
\end{proof}

\begin{thm}\label{Thm_Triv_p_MetaFrob_no_fusion}
	Let $G$ be a metabelian Frobenius group with cyclic complement $H$ of order~$m$ dividing $p-1$ and Frobenius kernel $F\cong C_p\times C_p$ such that the number of $G$-conjugacy classes of subgroups of $G$ of order $p$ is precisely $p+1$. Then, the trivial source character table $\Triv_{p}(G)=[T_{i,v}]_{1\leq i,v\leq p+3}$ seen as a block matrix is as given in \cref{table:tsct_no_fusion}. 
    More precisely,  the labelling of the rows and columns and the entries are as described below.
    \begin{enumerate}[label={\rm(\alph*)}]
        \item For every  $1\leq i,v\leq p+3$ the columns of $T_{i,v}$  may be labelled by the elements of $H$. 
        \item 
        The ordinary characters of the trivial source modules are as follows:
            \begin{enumerate}[label={\rm(\arabic*)}]
                   \item
                   $\{\chi^{}_{\widehat{M}}\mid  M\in \TS(G;Q_1)\}=\{\chi + \sum\limits_{j=m+1}^{m+(p+1)\cdot a(m)} {\chi_j}\mid \chi\in \Lin(G)\}$;
                   \item 
                   $\{\chi^{}_{\widehat{M}}\mid  M\in \TS(G;Q_i)\}=\{\lambda + \sum\limits_{\substack{\chi\in\Irr(G)\setminus\Lin(G)\\ \Ker(\chi)=Q_i}} \chi\mid \lambda\in \Lin(G)\}$ for every $2\leq i\leq p+2$;
                   \item 
                   $\{\chi^{}_{\widehat{M}}\mid  M\in \TS(G;Q_{p+3})\}=\Lin(G)$. 
                \end{enumerate}
        As $\Lin(G)=\Inf_{H}^G(\Irr(H))$ we choose the labelling of the rows of $T_{i,v}$ $(1\leq i,v\leq p+3)$ to match that of $X(H)$. 
  \item  With the labelling of the rows and of the columns given in {\rm(a)} and {\rm(b)}, we have:
    \begin{enumerate}[label={\rm(\arabic*)}]
		\item $T_{i,v}={\mathbf{0}}$ for every $2\leq v<i\leq p+2$ and for every $1\leq i< v\leq p+3$;
        \item $T_{1,1}=X(H) + \left(\begin{smallmatrix}
		p^2-1 & 0 & \cdots & 0\\
		\vdots & \vdots & \ddots & \vdots\\
		p^2-1 & 0 & \cdots & 0
		\end{smallmatrix}\right)$;
		\item $T_{2,1}=T_{i,1}=T_{i,i}=X(H) + \left(\begin{smallmatrix}
		p-1 & 0 & \cdots & 0\\
		\vdots & \vdots & \ddots & \vdots\\
		p-1 & 0 & \cdots & 0 
		\end{smallmatrix}\right)$  for every $2\leq i\leq p+2$;
		\item $T_{p+3,1} = T_{p+3,v} = X(H)$ for every $2\leq v\leq p+3$.
	\end{enumerate}
    \end{enumerate}
\end{thm}

\renewcommand*{\arraystretch}{1.4}
\begin{center}
 \begin{tiny}
\begin{table}[ht]
\begin{adjustbox}{max width=\textwidth -5pt}
\begin{tabularx}{\textwidth}{ |Y||Y||Y||Y||c||c|}

	  \hline 
	
	$T_{1,1}$
	& $\bf{0}$
	& $\bf{0}$
	& $\bf{0}$
	& $\cdots$
	& $\bf{0}$
	\\ \hline \hline

	$T_{2,1}$
	& $T_{2,2}=T_{2,1}$
	& $\bf{0}$
	& $\bf{0}$
	& $\cdots$
	& $\bf{0}$
	\\ \hline \hline

	$T_{3,1}=T_{2,1}$
	& $\bf{0}$
	& $T_{3,3}=T_{2,1}$
	& $\bf{0}$
	& $\cdots$
	& $\bf{0}$
	\\ \hline\hline
	
	$\vdots$
	& $\vdots$
	& $\smash{\rotatebox{20}{$\ddots$}}$
	& $\smash{\rotatebox{20}{$\ddots$}}$
	& $\smash{\rotatebox{20}{$\ddots$}}$
	& $\vdots$
	\\ \hline\hline

	$T_{p+2,1}=T_{2,1}$
	& $\bf{0}$
	& $\cdots$
	& $\bf{0}$
	& $T_{p+2,p+2}=T_{2,1}$
	& $\bf{0}$
	\\ \hline\hline

	$T_{p+3,1}$
	& \mbox{$T_{p+3,2}=T_{p+3,1}$}
	& $\bf{\cdots}$
	& \, $\bf{\cdots}$\, 
	& \mbox{$T_{p+3,p+2}=T_{p+3,1}$}
	& \mbox{$T_{p+3,p+3}=T_{p+3,1}$}
	\\ \hline
\end{tabularx}
\end{adjustbox}
	\caption{Trivial source character table $\Triv_p(G)$, seen as a block matrix.}
    \label{table:tsct_no_fusion}
\end{table}
 \end{tiny}
\end{center}

\normalsize{}

\begin{proof}
(a) Let $1\leq v\leq p+3$. Since $N_G(Q_v)=G$ by \cref{lem:Normalisers_MinimalFusionCase}(a), we may assume that $H$ is a set of representatives of the $p$-regular conjugacy classes of $N_G(Q_v)$ by \cref{prop:decompositionMatrix_Frobenius_groups_all_cases_relevant_for_us}(c). Thus, the claim follows from \cref{conv:tsctbl}. 
\par
\noindent (b) As in the proof of the previous case, because $G/Q_{p+3}\cong H$, $|\Lin(G)|=|G/[G,G]|$ and $[G,G]=Q_{p+3}$ by \cref{lem:Frobenius_General_Properties}(d), 
by abuse of notation we may write $\Lin(G)=\Inf_{H}^{G}(\Irr(H))$, giving the last claim.
Next, notice that by Subsection~\ref{ssec:DefTSCT},  \cref{lem:Normalisers_MinimalFusionCase} and \cref{prop:decompositionMatrix_Frobenius_groups_all_cases_relevant_for_us}, we have $|\TS(G;Q_v)|=|H|=|\Lin(G)|$ for every $1\leq v\leq p+3$. 
\begin{enumerate}[label=\rm(\arabic*), leftmargin=8mm]
    \item The ordinary characters of the PIMs of $kG$ can be read off from the decomposition matrix in \cref{table:decompositionMatrix_G_NO_FUSION}, that is, for each $1\leq a \leq m$ we have
 $$\Phi_{\varphi_a}= \chi_a + \sum\limits_{j=m+1}^{m+(p+1)\cdot a(m)} {\chi_j}$$
 where $\chi_a$ runs through $\Lin(G)$.
    \item 
    Fix $2\leq i\leq p+2$.  We know from \cref{lem:Normalisers_MinimalFusionCase} that $N_G(Q_i)=G$ and $\overline{N}_G(Q_i)=G/Q_i$ is a Frobenius group with Frobenius complement~$H= Q_iH/Q_i$ and Frobenius kernel $F/Q_i$. 
    First, it follows from the bijections and the notation introduced in \cref{prop:decompositionMatrix_G_NO_FUSION} and in Subsection~\ref{ssec:DefTSCT} 
    that    
        \[
            \TS(G;Q_i)=\{\Inf_{G/Q_i}^{G} (P_{\psi})\mid \psi\in\IBr_p(G/Q_i)\}\,.
        \]
    Then, it follows from \cref{prop:decompositionMatrix_G_NO_FUSION}(b) that for every $1\leq u\leq m$ the ordinary characters of the PIMs $P_{\psi_u}$ of $k[G/Q_i]$ are given by~
        \[
            \Phi_{\psi_u}=\theta_u + \sum_{b=1}^{a(m)} {\theta_{m+b}}\,,
        \]
    where $\theta_u\in\Lin(G/Q_i)$ and $\theta_{m+b}\in \Irr(G/Q_i)\setminus\Lin(G/Q_i)$ with $\ker(\theta_{m+b})=Q_i$ for each {$1\leq b\leq a(m)$}\,.
    Next we observe that
        \[
            \Inf_{G/Q_i}^{G}\big(\Irr(G/Q_i)\setminus\Lin(G/Q_i)\big) =\{\chi\in \Irr(G)\setminus\Lin(G)\mid \Ker(\chi)=Q_i\}\,.
        \]
    Indeed, as the kernels of the non-trivial characters of $F$  are cyclic of order $p$ and normal in $G$, it follows that the kernel of any the non-linear irreducible character of $G$ is equal to the kernel of the character(s) it is induced from, hence also cyclic of order $p$.
    Therefore, we obtain that 
    \begin{equation*}
        \begin{split}
            \Inf_{G/Q_i}^{G}(\Phi_{\psi_u})=\Inf_{G/Q_i}^{G}(\theta_u + \sum_{b=1}^{a(m)} {\theta_{m+b}}) & =\Inf_{G/Q_i}^{G}(\theta_u)+\sum_{b=1}^{a(m)}\Inf_{G/Q_i}^{G}({\theta_{m+b}})\\
            &= \Inf_{G/Q_i}^{G}(\theta_u)+\sum_{\substack{\chi\in\Irr(G)\setminus\Lin(G)\\\ker(\chi)=Q_i}}\chi
        \end{split}
    \end{equation*}
    for each $1\leq u\leq m$. Clearly, the characters $\Inf_{G/Q_i}^{G}(\theta_u)$ run through $\Lin(G)$ when $u$ runs from $1$ to $m$, and the claim follows.
    \item The claim follows from \cref{lem:Simples_G_mod_Op(G)}(b), where equality holds by the argument above.
\end{enumerate} 

\noindent (c) We now compute the entries of $\Triv_p(G)$. 
\begin{enumerate}[label=\rm(\arabic*), leftmargin=8mm]
    \item The assertion is immediate from \cref{lem:PropTrivpG}(a).   
    \item By \cref{lem:PropTrivpG}(b), we have $T_{1,1}= \Phi_p(G)$. 
    It is clear from \cref{prop:decompositionMatrix_Frobenius_groups_all_cases_relevant_for_us} and Part (b)(1) that  the degree of the characters $\Phi_{\varphi_a}$ ($1\leq a\leq m$) is $1+p^2-1$. Therefore, it is now only left to prove that $\chi_j(y)=0$ for all $m+1\leq j\leq m+(p+1)\cdot a(m)$ and for all~$y\in H\setminus\{1\}$. But this is clear since for any $\nu\in\Irr(F)$, we have
	$$(\Ind_F^G(\nu))(y) = \frac{1}{|F|}\sum\limits_{\substack{g\in G\\ gyg^{-1}\in F}}^{} \nu (gyg^{-1})=0.$$
\bigskip
	\item \underline{The matrices $T_{i,i}\ (2\leq i\leq p+2)$}. Fix $i\in \{2, \ldots , p+2\}$.
    By \cref{lem:PropTrivpG}(b), we have $T_{i,i}= \Phi_p(\overline{N}_G(Q_i))$. 
    By \cref{lem:Normalisers_MinimalFusionCase} the group $\overline{N}_G(Q_i)=G/Q_i$ is a Frobenius group with Frobenius complement~$H$ and Frobenius kernel $F/Q_i$. The ordinary characters of the PIMs of $k\overline{N}_G(Q_i)$ can be read off from \cref{table:decompositionMatrix_N_bar_NO_FUSION}, namely
    $$\Phi_{\psi_a}=\theta_a + \sum_{j=1}^{\frac{p-1}{m}} {\theta_{m+j}}$$
for each $1\leq a\leq m$.
As any element of  $\Irr(G/Q_i)\setminus\Lin(G/Q_i)$ is induced from a linear character of $F/Q_i$, its degree is $[G/Q_i : F/Q_i]\cdot 1 = [G:F]=|H|=m$. Therefore, $\textup{deg}(\Phi_{\psi_a})=1+\frac{p-1}{m}\cdot m = 1+p-1=p$ for every $1\leq a\leq m$. As in (c)(2), using the formula for induced characters, we see that all non-linear constituents of $\Phi_{\psi_a}$ evaluate to $0$ at all $g\in H\setminus \{1\}$. The claim follows, as all the linear characters of $G/Q_i$ are precisely the inflations of the characters in $\Irr(H)$. This proves that 
$$T_{i,i}=X(H) + \left(\begin{smallmatrix}
		p-1 & 0 & \cdots & 0\\
		\vdots & \vdots & \ddots & \vdots\\
		p-1 & 0 & \cdots & 0 
		\end{smallmatrix}\right)\,.$$ 
  \bigskip

\noindent \underline{The matrices $T_{i,1}\ (2\leq i\leq p+2)$}. Fix $i\in \{2, \ldots , p+2\}$. By \cref{lem:PropTrivpG}(c), we have
 $$T_{i,1}=\big( \chi^{}_{\widehat{M}}(s)\big)_{M\in \TS(G;Q_{i}), s\in [G]_{p^\prime} }.$$
For any  $M\in TS(G;Q_i)$, by the bijections in Subsection \ref{ssec:DefTSCT}, we have $$\chi_{\widehat{M}}=\Ind_{N_G(Q_i)}^{G}\Inf_{N_G(Q_i)/Q_i}^{N_G(Q_i)} (\Phi_{\psi_a}) = \Ind_{G}^{G}\Inf_{G/Q_i}^{G} (\Phi_{\psi_a}) = \Inf_{G/Q_i}^{G} (\Phi_{\psi_a})$$
for a unique $a\in\{1,\ldots, m\}$. Since we set in (a) that $[G]_{p^\prime}=H$, it follows  immediately that $T_{i,1}=T_{i,i}$. Moreover, we have $T_{i,i}=T_{2,2}=T_{2,1}$.
\bigskip
	\item Fix $v\in \{1, \ldots , p+3\}$. By definition, we have 
$$T_{p+3,v} = \big[ \tau_{Q_{v},s}^{G}([M])\big]_{M\in \TS(G;Q_{p+3}), s\in [\overline{N}_G(Q_v)]_{p^\prime} }\,.$$
So, let $M\in\TS(G;Q_{p+3})$ and let $t\in [N_G(Q_v)]_{p^\prime}$. Recall that by \cref{lem:Normalisers_MinimalFusionCase}(a) we have $N_G(Q_v)=G$ and by \cref{prop:decompositionMatrix_Frobenius_groups_all_cases_relevant_for_us}(c) we may choose $[N_G(Q_v)]_{p^\prime} = H$. By definition, 
$$\tau_{Q_v, t}^G([M]) = \varphi_{M[Q_v]}(t).$$
Since $Q_v\leq Q_{p+3}=F \unlhd G$, it follows from  \cite[Proposition 5.10.4]{Linckelmann1} that $M[Q_v]=\Res^G_{N_G(Q_v)}(M)$. 
Hence, $\varphi_{M[Q_v]}(t) = \chi_{\widehat{M}}(t)$ by \cref{lem:Simples_G_mod_Op(G)}(b), and moreover $T_{p+3,1} = T_{p+3,v} = X(H)$. 
	\end{enumerate}
\end{proof}

\begin{exmp}
Let $G$ be the Frobenius group $(C_5\times C_5)\rtimes C_4$ of order $100$ with minimal fusion pattern, i.e. $G$ has precisely $6$ distinct conjugacy classes of subgroups of order $5$. It follows that we have $8$ conjugacy classes of $5$-subgroups of $G$, namely:
$$Q_1=\{1\},\quad Q_2\cong Q_3\cong Q_4\cong Q_5\cong Q_6\cong Q_7 \cong C_5,\quad Q_8\cong C_5\times C_5.$$
Notice that $G$ is isomorphic to the group labelled by\ [ 100, 11 ]\ in GAP's SmallGroups library, see~\cite{gap}. The ordinary character table of $G$ is as given in \cref{Ordinary_ct_FrobGroup_C5_C5_C4}, where 
$\zeta_4 := \exp(\frac{2\pi i}{4})$.
\enlargethispage{15mm}
\begin{center}
\begin{table}[h]
\scalebox{0.84}{
\begin{tabular}{@{}l@{}l@{}l@{}}
\hline
\(\begin{array}{|l|cccccccccc|}
  & 1a & 4a & 2a & 4b & 5a & 5b & 5c & 5d & 5e & 5f\\ \hline
\chi_{1} & 1 & 1 & 1 & 1 & 1 & 1 & 1 & 1 & 1 & 1\\
\chi_{2} & 1 & \zeta_4 & -1 & -\zeta_4 & 1 & 1  & 1 & 1 & 1 & 1\\
\chi_{3} & 1 & -1 & 1 & -1 & 1 & 1  & 1 & 1 & 1 & 1\\
\chi_{4} & 1 & -\zeta_4 & -1 & \zeta_4 & 1 & 1  & 1 & 1 & 1 & 1\\
\chi_{5} & 4 & 0 & 0 & 0 & 4 & -1  & -1 & -1 & -1 & -1\\
\chi_{6} & 4 & 0 & 0 & 0 & -1 & 4  & -1 & -1 & -1 & -1\\
\chi_{7} & 4 & 0 & 0 & 0 & -1 & -1 & 4 & -1 & -1 & -1\\
\chi_{8} & 4 & 0 & 0 & 0 & -1 & -1  & -1 & 4 & -1 & -1\\
\chi_{9} & 4 & 0 & 0 & 0 & -1 & -1 & -1 & -1 & 4 & -1\\
\chi_{10} & 4 & 0 & 0 & 0 & -1 & -1 & -1 & -1 & -1 & 4\\
\hline
\end{array}\)\\
\end{tabular}
}
\title{}
\caption{Ordinary character table of $(C_5\times C_5)\rtimes C_4$.}
\label{Ordinary_ct_FrobGroup_C5_C5_C4}
\end{table}
\end{center}
The trivial source character table $\Triv_5(G)$ is as given in \cref{TSCT_FrobGrp_100}. Following our conventions, we label the columns of $\Triv_5(G)$ with $5^\prime$-elements in $N_G(Q_v)$ instead of $\overline{N}_G(Q_v)$\ $(1\leq v\leq 8)$. 

\begin{landscape}
\vspace*{\fill}
\begin{center}
\begin{table}[h]
\resizebox{20cm}{!}{
\begin{tabular}{@{}l@{}l@{}l@{}l@{}l@{}l@{}l@{}l@{}l@{}l@{}l@{}l@{}l@{}l@{}l@{}l@{}l@{}l@{}l@{}l@{}}
\(\begin{array}{|l|cccc|cccc|cccc|cccc|cccc|cccc|cccc|cccc|}
\hline
\textup{Normalisers}\ N_v & \multicolumn{4}{c|}{N_{1}\cong (C_5\times C_5)\rtimes C_4} & \multicolumn{4}{c|}{N_{2}\cong (C_5\times C_5)\rtimes C_4} & \multicolumn{4}{c|}{N_{3}\cong (C_5\times C_5)\rtimes C_4} & \multicolumn{4}{c|}{N_{4}\cong (C_5\times C_5)\rtimes C_4} & \multicolumn{4}{c|}{N_{5}\cong (C_5\times C_5)\rtimes C_4} & \multicolumn{4}{c|}{N_{6}\cong (C_5\times C_5)\rtimes C_4} & \multicolumn{4}{c|}{N_{7}\cong (C_5\times C_5)\rtimes C_4} & \multicolumn{4}{c|}{N_{8}\cong (C_5\times C_5)\rtimes C_4}\\ \hline
\textup{Representatives}\ n_j\ \in\ N_i & 1a & 4a & 2a & 4b & 1a & 4a & 2a & 4b & 1a & 4a & 2a & 4b & 1a & 4a & 2a & 4b & 1a & 4a & 2a & 4b & 1a & 4a & 2a & 4b & 1a & 4a & 2a & 4b & 1a & 4a & 2a & 4b\\ \hline
\chi_{1} + \chi_{5} + \chi_{6} + \chi_{7} + \chi_{8} + \chi_{9} + \chi_{10} & 25 & 1 & 1 & 1 & 0 & 0 & 0 & 0 & 0 & 0 & 0 & 0 & 0 & 0 & 0 & 0 & 0 & 0 & 0 & 0 & 0 & 0 & 0 & 0 & 0 & 0 & 0 & 0 & 0 & 0 & 0 & 0\\
\chi_{2} + \chi_{5} + \chi_{6} + \chi_{7} + \chi_{8} + \chi_{9} + \chi_{10} & 25 & \zeta_4 & -1 & -\zeta_4 & 0 & 0 & 0 & 0 & 0 & 0 & 0 & 0 & 0 & 0 & 0 & 0 & 0 & 0 & 0 & 0 & 0 & 0 & 0 & 0 & 0 & 0 & 0 & 0 & 0 & 0 & 0 & 0\\
\chi_{3} + \chi_{5} + \chi_{6} + \chi_{7} + \chi_{8} + \chi_{9} + \chi_{10} & 25 & -1 & 1 & -1 & 0 & 0 & 0 & 0 & 0 & 0 & 0 & 0 & 0 & 0 & 0 & 0 & 0 & 0 & 0 & 0 & 0 & 0 & 0 & 0 & 0 & 0 & 0 & 0 & 0 & 0 & 0 & 0\\
\chi_{4} + \chi_{5} + \chi_{6} + \chi_{7} + \chi_{8} + \chi_{9} + \chi_{10} & 25 & -\zeta_4 & -1 & \zeta_4 & 0 & 0 & 0 & 0 & 0 & 0 & 0 & 0 & 0 & 0 & 0 & 0 & 0 & 0 & 0 & 0 & 0 & 0 & 0 & 0 & 0 & 0 & 0 & 0 & 0 & 0 & 0 & 0\\
\hline
\chi_{1} + \chi_{5} & 5 & 1 & 1 & 1 & 5 & 1 & 1 & 1 & 0 & 0 & 0 & 0 & 0 & 0 & 0 & 0 & 0 & 0 & 0 & 0 & 0 & 0 & 0 & 0 & 0 & 0 & 0 & 0 & 0 & 0 & 0 & 0\\
\chi_{2} + \chi_{5} & 5 & \zeta_4 & -1 & -\zeta_4 & 5 & \zeta_4 & -1 & -\zeta_4 & 0 & 0 & 0 & 0 & 0 & 0 & 0 & 0 & 0 & 0 & 0 & 0 & 0 & 0 & 0 & 0 & 0 & 0 & 0 & 0 & 0 & 0 & 0 & 0\\
\chi_{3} + \chi_{5} & 5 & -1 & 1 & -1 & 5 & -1 & 1 & -1 & 0 & 0 & 0 & 0 & 0 & 0 & 0 & 0 & 0 & 0 & 0 & 0 & 0 & 0 & 0 & 0 & 0 & 0 & 0 & 0 & 0 & 0 & 0 & 0\\
\chi_{4} + \chi_{5} & 5 & -\zeta_4 & -1 & \zeta_4 & 5 & -\zeta_4 & -1 & \zeta_4 & 0 & 0 & 0 & 0 & 0 & 0 & 0 & 0 & 0 & 0 & 0 & 0 & 0 & 0 & 0 & 0 & 0 & 0 & 0 & 0 & 0 & 0 & 0 & 0\\
 \hline
\chi_{1} + \chi_{6} & 5 & 1 & 1 & 1 & 0 & 0 & 0 & 0 & 5 & 1 & 1 & 1 & 0 & 0 & 0 & 0 & 0 & 0 & 0 & 0 & 0 & 0 & 0 & 0 & 0 & 0 & 0 & 0 & 0 & 0 & 0 & 0\\
\chi_{2} + \chi_{6} & 5 & \zeta_4 & -1 & -\zeta_4 & 0 & 0 & 0 & 0 & 5 & \zeta_4 & -1 & -\zeta_4 & 0 & 0 & 0 & 0 & 0 & 0 & 0 & 0 & 0 & 0 & 0 & 0 & 0 & 0 & 0 & 0 & 0 & 0 & 0 & 0\\
\chi_{3} + \chi_{6} & 5 & -1 & 1 & -1 & 0 & 0 & 0 & 0 & 5 & -1 & 1 & -1 & 0 & 0 & 0 & 0 & 0 & 0 & 0 & 0 & 0 & 0 & 0 & 0 & 0 & 0 & 0 & 0 & 0 & 0 & 0 & 0\\
\chi_{4} + \chi_{6} & 5 & -\zeta_4 & -1 & \zeta_4 & 0 & 0 & 0 & 0 & 5 & -\zeta_4 & -1 & \zeta_4 & 0 & 0 & 0 & 0 & 0 & 0 & 0 & 0 & 0 & 0 & 0 & 0 & 0 & 0 & 0 & 0 & 0 & 0 & 0 & 0\\
 \hline
\chi_{1} + \chi_{7} & 5 & 1 & 1 & 1 & 0 & 0 & 0 & 0 & 0 & 0 & 0 & 0 & 5 & 1 & 1 & 1 & 0 & 0 & 0 & 0 & 0 & 0 & 0 & 0 & 0 & 0 & 0 & 0 & 0 & 0 & 0 & 0\\
\chi_{2} + \chi_{7} & 5 & \zeta_4 & -1 & -\zeta_4 & 0 & 0 & 0 & 0 & 0 & 0 & 0 & 0 & 5 & \zeta_4 & -1 & -\zeta_4 & 0 & 0 & 0 & 0 & 0 & 0 & 0 & 0 & 0 & 0 & 0 & 0 & 0 & 0 & 0 & 0\\
\chi_{3} + \chi_{7} & 5 & -1 & 1 & -1 & 0 & 0 & 0 & 0 & 0 & 0 & 0 & 0 & 5 & -1 & 1 & -1 & 0 & 0 & 0 & 0 & 0 & 0 & 0 & 0 & 0 & 0 & 0 & 0 & 0 & 0 & 0 & 0\\
\chi_{4} + \chi_{7} & 5 & -\zeta_4 & -1 & \zeta_4 & 0 & 0 & 0 & 0 & 0 & 0 & 0 & 0 & 5 & -\zeta_4 & -1 & \zeta_4 & 0 & 0 & 0 & 0 & 0 & 0 & 0 & 0 & 0 & 0 & 0 & 0 & 0 & 0 & 0 & 0\\
 \hline
\chi_{1} + \chi_{8} & 5 & 1 & 1 & 1 & 0 & 0 & 0 & 0 & 0 & 0 & 0 & 0 & 0 & 0 & 0 & 0 & 5 & 1 & 1 & 1 & 0 & 0 & 0 & 0 & 0 & 0 & 0 & 0 & 0 & 0 & 0 & 0\\
\chi_{2} + \chi_{8} & 5 & \zeta_4 & -1 & -\zeta_4 & 0 & 0 & 0 & 0 & 0 & 0 & 0 & 0 & 0 & 0 & 0 & 0 & 5 & \zeta_4 & -1 & -\zeta_4 & 0 & 0 & 0 & 0 & 0 & 0 & 0 & 0 & 0 & 0 & 0 & 0\\
\chi_{3} + \chi_{8} & 5 & -1 & 1 & -1 & 0 & 0 & 0 & 0 & 0 & 0 & 0 & 0 & 0 & 0 & 0 & 0 & 5 & -1 & 1 & -1 & 0 & 0 & 0 & 0 & 0 & 0 & 0 & 0 & 0 & 0 & 0 & 0\\
\chi_{4} + \chi_{8} & 5 & -\zeta_4 & -1 & \zeta_4 & 0 & 0 & 0 & 0 & 0 & 0 & 0 & 0 & 0 & 0 & 0 & 0 & 5 & -\zeta_4 & -1 & \zeta_4 & 0 & 0 & 0 & 0 & 0 & 0 & 0 & 0 & 0 & 0 & 0 & 0\\
 \hline
\chi_{1} + \chi_{9} & 5 & 1 & 1 & 1 & 0 & 0 & 0 & 0 & 0 & 0 & 0 & 0 & 0 & 0 & 0 & 0 & 0 & 0 & 0 & 0 & 5 & 1 & 1 & 1 & 0 & 0 & 0 & 0 & 0 & 0 & 0 & 0\\
\chi_{2} + \chi_{9} & 5 & \zeta_4 & -1 & -\zeta_4 & 0 & 0 & 0 & 0 & 0 & 0 & 0 & 0 & 0 & 0 & 0 & 0 & 0 & 0 & 0 & 0 & 5 & \zeta_4 & -1 & -\zeta_4 & 0 & 0 & 0 & 0 & 0 & 0 & 0 & 0\\
\chi_{3} + \chi_{9} & 5 & -1 & 1 & -1 & 0 & 0 & 0 & 0 & 0 & 0 & 0 & 0 & 0 & 0 & 0 & 0 & 0 & 0 & 0 & 0 & 5 & -1 & 1 & -1 & 0 & 0 & 0 & 0 & 0 & 0 & 0 & 0\\
\chi_{4} + \chi_{9} & 5 & -\zeta_4 & -1 & \zeta_4 & 0 & 0 & 0 & 0 & 0 & 0 & 0 & 0 & 0 & 0 & 0 & 0 & 0 & 0 & 0 & 0 & 5 & -\zeta_4 & -1 & \zeta_4 & 0 & 0 & 0 & 0 & 0 & 0 & 0 & 0\\
 \hline
\chi_{1} + \chi_{10} & 5 & 1 & 1 & 1 & 0 & 0 & 0 & 0 & 0 & 0 & 0 & 0 & 0 & 0 & 0 & 0 & 0 & 0 & 0 & 0 & 0 & 0 & 0 & 0 & 5 & 1 & 1 & 1 & 0 & 0 & 0 & 0\\
\chi_{2} + \chi_{10} & 5 & \zeta_4 & -1 & -\zeta_4 & 0 & 0 & 0 & 0 & 0 & 0 & 0 & 0 & 0 & 0 & 0 & 0 & 0 & 0 & 0 & 0 & 0 & 0 & 0 & 0 & 5 & \zeta_4 & -1 & -\zeta_4 & 0 & 0 & 0 & 0\\
\chi_{3} + \chi_{10} & 5 & -1 & 1 & -1 & 0 & 0 & 0 & 0 & 0 & 0 & 0 & 0 & 0 & 0 & 0 & 0 & 0 & 0 & 0 & 0 & 0 & 0 & 0 & 0 & 5 & -1 & 1 & -1 & 0 & 0 & 0 & 0\\
\chi_{4} + \chi_{10} & 5 & -\zeta_4 & -1 & \zeta_4 & 0 & 0 & 0 & 0 & 0 & 0 & 0 & 0 & 0 & 0 & 0 & 0 & 0 & 0 & 0 & 0 & 0 & 0 & 0 & 0 & 5 & -\zeta_4 & -1 & \zeta_4 & 0 & 0 & 0 & 0\\
 \hline
\chi_{1} & 1 & 1 & 1 & 1 & 1 & 1 & 1 & 1 & 1 & 1 & 1 & 1 & 1 & 1 & 1 & 1 & 1 & 1 & 1 & 1 & 1 & 1 & 1 & 1 & 1 & 1 & 1 & 1 & 1 & 1 & 1 & 1\\
\chi_{2} & 1 & \zeta_4 & -1 & -\zeta_4 & 1 & \zeta_4 & -1 & -\zeta_4 & 1 & \zeta_4 & -1 & -\zeta_4 & 1 & \zeta_4 & -1 & -\zeta_4 & 1 & \zeta_4 & -1 & -\zeta_4 & 1 & \zeta_4 & -1 & -\zeta_4 & 1 & \zeta_4 & -1 & -\zeta_4 & 1 & \zeta_4 & -1 & -\zeta_4\\
\chi_{3} & 1 & -1 & 1 & -1 & 1 & -1 & 1 & -1 & 1 & -1 & 1 & -1 & 1 & -1 & 1 & -1 & 1 & -1 & 1 & -1 & 1 & -1 & 1 & -1 & 1 & -1 & 1 & -1 & 1 & -1 & 1 & -1\\
\chi_{4} & 1 & -\zeta_4 & -1 & \zeta_4 & 1 & -\zeta_4 & -1 & \zeta_4 & 1 & -\zeta_4 & -1 & \zeta_4 & 1 & -\zeta_4 & -1 & \zeta_4 & 1 & -\zeta_4 & -1 & \zeta_4 & 1 & -\zeta_4 & -1 & \zeta_4 & 1 & -\zeta_4 & -1 & \zeta_4 & 1 & -\zeta_4 & -1 & \zeta_4\\
\hline

\end{array}\)\end{tabular}
}
\caption{Trivial source character table of $(C_5\times C_5)\rtimes C_4$ at $p=5$.}
\label{TSCT_FrobGrp_100}
\title{}
\end{table}
\end{center}
\vspace*{\fill}
\end{landscape}
\end{exmp}

\newpage

\textbf{Acknowledgments.}
The research in this article started as part of Project A18 of the DFG Collaborative Research Centre SFB/TRR 195, by which the authors gratefully acknowledge financial support. They are also grateful for detailed comments received on a preliminary version of this manuscript, in particular for the finding of an error in the proof of Proposition \ref{prop:FrobeniusGroups_max_fusion_and_no_fusion}. 
Finally, they are grateful to Richard Parker, who first mentioned to them the problem of calculating the species tables of the trivial source ring during a \emph{Research Cambridge Style} session of the Nikolaus Conference 2014. This article is dedicated to his memory.	
	


	\nocite{}
	\bibliographystyle{aomalpha}
	\bibliography{biblio.bib}
	



\appendix
\section{$\Triv_p(\AGL_1(p^2))$ in table form}\label{app:tables}%

To support intuition, in this appendix, we give $\Triv_p(\AGL_1(p^2))$ described in \cref{Thm_Triv_p_AGL_1_p_squared} in \emph{table form}, where
$\zeta := \exp(\frac{2\pi i}{p^2-1})$.

 
	\renewcommand*{\arraystretch}{1.5}

\begin{small}
\begin{center}
\begin{adjustbox}{angle=90, max height=\dimexpr\textheight-5\baselineskip}
\begin{minipage}{\textheight}   

\centering

\begin{tabular}{c|c||c|c|c|c|c|c|c|c|}

	  \cline{2-10}
		
		&&

			$1$

		& 

			$h$

		&

			$\cdots $

		&

            \begin{tabular}{c}
			$h^{b(p+1)-1}$\\
                $(1\leq b\leq p-2)$
            \end{tabular}
            
		&
  
            \begin{tabular}{c}
			$h^{b(p+1)}$\\
                $(1\leq b\leq p-2)$
            \end{tabular}
            
		&
            \begin{tabular}{c}
			$h^{b(p+1)+1}$\\
                $(1\leq b\leq p-2)$
            \end{tabular}
		&

			$\cdots $

		&

			$h^{p^2-2}$

		\\ \hline \hline

		\multirow{4}{*}[0ex]{
			$T_{1,1} $}
		& $\chi_1 + \chi_{p^2}$ 
		& $p^2$
		& $1$ 
		& $\cdots$ 
		& $1$ 
		& $1$ 
		& $1$ 
		& $\cdots$
		& $1$
		\\ \cline{2-10}

		& $\chi_2 + \chi_{p^2}$ 
		& $p^2$
		& $\zeta$
		& $\cdots$ 
		& $\zeta^{b(p+1)-1}$ 
		& $\zeta^{b(p+1)}$ 
		& $\zeta^{b(p+1)+1}$ 
		& $\cdots$
		& $\zeta^{p^2-2}$
		\\ \cline{2-10}

		&\begin{tabular}{c}
			$\vdots$
			
		\end{tabular}
		& $\vdots$
		& $\vdots$
		& $\cdots$ 
		& $\vdots$ 
		& $\vdots$ 
		& $\vdots$ 
		& $\cdots$
		& $\vdots$
		\\ \cline{2-10}

		&\begin{tabular}{c}
			$\chi_{p^2-1} + \chi_{p^2}$
		\end{tabular}
		& $p^2$
		& $\zeta^{p^2-2}$
		& $\cdots$
		& $\zeta^{(b(p+1)-1)\cdot (p^2-2)}$ 
		& $\zeta^{b(p+1)\cdot (p^2-2)}$ 
		& $\zeta^{(b(p+1)+1)\cdot (p^2-2)}$ 
		& $\cdots$ 
		& $\zeta^{(p^2-2)^2}$ 
		\\ 
		\hline
		\hline

		\multirow{4}{*}[0ex]{$T_{2,1} $}
		& $\chi_{p^2} + \sum_{a=0}^{p} {\chi_{a(p-1)+1}}$
		& $p(p+1)$
		& $0$
		& $\cdots$
		& $0$ 
		& $p+1$ 
		& $0$ 
		& $\cdots$ 
		& $0$
		\\ \cline{2-10}

		& $\chi_{p^2} + \sum_{a=0}^{p} {\chi_{a(p-1)+2}}$
		& $p(p+1)$
		& $0$
		& $\cdots$
		& $0$ 
		& $\zeta^{b(p+1)}\cdot (p+1)$ 
		& $0$ 
		& $\cdots$ 
		& $0$
		\\ \cline{2-10}

		&\begin{tabular}{c}
			$\vdots$
			
		\end{tabular}
		& $\vdots$
		& $\vdots$
		& $\cdots$
		& $\vdots$
		& $\vdots$
		& $\vdots$
		& $\cdots$
		& $\vdots$
		\\ \cline{2-10}

		&\begin{tabular}{c}
			$\chi_{p^2} + \sum_{a=0}^{p} {\chi_{a(p-1)+(p-1)}}$
		\end{tabular}
		& $p(p+1)$
		& $0$
		& $\cdots$
		& $0$ 
		& $\zeta^{(p-2)\cdot b(p+1)}\cdot(p+1)$ 
		& $0$ 
		& $\cdots$ 
		& $0$
		\\ 
		\hline
		\hline

		\multirow{4}{*}[0ex]{
			$T_{3,1} $}
		& $\chi_1$
		& $1$
		& $1$
		& $\cdots$
		& $1$ 
		& $1$ 
		& $1$ 
		& $\cdots$ 
		& $1$
		\\ \cline{2-10}

		& $\chi_2$
		& $1$
		& $\zeta$ 
		& $\cdots$
		& $\zeta^{b(p+1)-1}$ 
		& $\zeta^{b(p+1)}$ 
		& $\zeta^{b(p+1)+1}$ 
		& $\cdots$ 
		& $\zeta^{p^2-2}$
		\\ \cline{2-10}

		&\begin{tabular}{c}
			$\vdots$
			
		\end{tabular}
		& $\vdots$
		& $\vdots$
		& $\cdots$
		& $\vdots$ 
		& $\vdots$
		& $\vdots$
		& $\cdots$ 
		& $\vdots$
		\\ \cline{2-10}

		&\begin{tabular}{c}
			$\chi_{p^2-1}$
		\end{tabular}
		& $1$
		& $\zeta^{p^2-2}$
		& $\cdots$
		& $\zeta^{(b(p+1)-1)\cdot (p^2-2)}$ 
		& $\zeta^{b(p+1)\cdot (p^2-2)}$ 
		& $\zeta^{(b(p+1)+1)\cdot (p^2-2)}$ 
		& $\cdots$ 
		& $\zeta^{(p^2-2)^2}$ 
		\\ 
		\hline
		
		\end{tabular}
		
\captionof{table}{$T_{i,1}$ for $1 \leq i \leq 3$ (first block column of $\Triv_p(\AGL_1(p^2)$).}
\label{table:TSCT_AGL_1_p_squared_T_i_1}

\end{minipage}
\end{adjustbox}
\end{center}

\end{small}

\newpage


\begin{center}
			\renewcommand*{\arraystretch}{1.5}
        \begin{table}
			\begin{tabular}{c|c||c|c|c|c|c|c|}
				 \cline{2-7}
				
    &
				&
				\begin{tabular}{c}
					\textbf{$1$}
				\end{tabular}
				& 
				\begin{tabular}{c}
					$h^{p+1}$
				\end{tabular}
				&
				\begin{tabular}{c}
					$h^{2\cdot (p+1)}$ 
				\end{tabular}
				&
					$\cdots$ 
				&
				\begin{tabular}{c}
					$h^{(p-2)\cdot (p+1)}$ 
				\end{tabular}
				\\ \hline \hline 			
			\multirow{4}{*}[0ex]{
			$T_{2,2} $}
   &
				$\chi_{p^2} + \sum_{a=0}^{p} {\chi_{a(p-1)+1}}$
				& $p$
				& $1$
				& $1$
				& $\cdots$
				& $1$
				\\ \cline{2-7}

	&			
				$\chi_{p^2} + \sum_{a=0}^{p} {\chi_{a(p-1)+2}}$
				& $p$
				& $\zeta^{p+1}$
				& $\zeta^{2\cdot (p+1)}$
				& $\cdots$
				& $\zeta^{(p-2)\cdot (p+1)}$
				\\ \cline{2-7}

	&			
				\begin{tabular}{c}
					$\vdots$
				\end{tabular}
				& $\vdots$
				& $\vdots$
				& $\vdots$
				& $\cdots$
				& $\vdots$
				\\ \cline{2-7}			
	&			
				$\chi_{p^2} + \sum_{a=0}^{p} {\chi_{a(p-1)+p-1}}$
				& $p$
				& $\zeta^{(p-2)\cdot (p+1)}$
				& $\zeta^{2\cdot (p-2)\cdot (p+1)}$
				& $\cdots$
				& $\zeta^{{(p-2)}^2\cdot (p+1)}$
				
				\\ \hline \hline 
						\multirow{5}{*}[0ex]{
			$T_{3,2} $}
   &
				$\chi_1$
				& $1$
				& $1$
				& $1$
				& $\cdots$
				& $1$
				\\ \cline{2-7}	

    &
				$\chi_2$
				& $1$
				& $\zeta^{(p+1)}$
				& $\zeta^{2\cdot (p+1)}$
				& $\cdots$
				& $\zeta^{(p-2)\cdot (p+1)}$
				\\ \cline{2-7}	
	&			
				$\chi_3$
				& $1$
				& $\zeta^{2\cdot (p+1)}$
				& $\zeta^{4\cdot (p+1)}$
				& $\cdots$
				& $\zeta^{2\cdot (p-2)\cdot (p+1)}$
				\\ \cline{2-7}	
	&			
				$\vdots$
				& $\vdots$
				& $\vdots$
				& $\vdots$
				& $\cdots$
				& $\vdots$
				\\ \cline{2-7}
	&			
				$\chi_{p^2-1}$
				& $1$
				& $\zeta^{(p^2-2)\cdot (p+1)}$
				& $\zeta^{2\cdot (p^2-2)\cdot (p+1)}$
				& $\cdots$
				& $\zeta^{(p^2-2)\cdot (p-2)\cdot (p+1)}$
				
				\\ \hline 
    			\end{tabular}
				\caption{$T_{i,2}$ for $2 \leq  i \leq 3$.} 
				\label{table:TSCT_AGL_1_p_squared_T_2_2_and_T_3_2}				
\end{table}
\end{center}
   	\begin{center}
			\renewcommand*{\arraystretch}{1.5}
            \begin{table}[h]
			\begin{tabular}{c|c||c|}

				 \cline{2-3}
				
				&

				& 
				
					$h^j\ (0\leq j \leq p^2-2)$
				
				\\ \hline \hline 			
				
				$T_{3,3}$ & 
					$\chi_a\ (1\leq a \leq p^2-1)$
				
				& $\zeta^{(a-1)j}\ $\\ \hline
                \end{tabular}
				\caption{$T_{3,3}$.} 
				\label{table:TSCT_AGL_1_p_squared_T_3_3}
		\end{table}
	\end{center}


\vfill 

\end{document}